\documentclass{aims} 
\usepackage{amsmath}
\usepackage{paralist}
\usepackage[misc]{ifsym}
\usepackage{epsfig} 
\usepackage{epstopdf} 
\usepackage[colorlinks=true]{hyperref}

\usepackage{dsfont}
\usepackage{amssymb}

\hypersetup{urlcolor=blue, citecolor=red}
\allowdisplaybreaks

\def\currentvolume{X}
 \def\currentissue{X}
  \def\currentyear{200X}
   \def\currentmonth{XX}
    \def\ppages{X-XX}
     \def\DOI{10.3934/xx.xxxxxxx}

\newtheorem{theorem}{Theorem}[section]
\newtheorem{corollary}[theorem]{Corollary}

\newtheorem{lemma}[theorem]{Lemma}

\theoremstyle{definition}
\newtheorem{definition}[theorem]{Definition}
\newtheorem{remark}[theorem]{Remark}

\newcommand{\ave}[1]{ \left[ #1 \right]}

\newcommand{\thm}[1]{Theorem~\ref{#1}}

\newcommand{\sect}[1]{Section~\ref{#1}}

\newcommand{\jap}[1]{\left\langle #1 \right\rangle}

\def \b {\beta}
\def \g {\gamma}
\def \d {\delta}
\def \e {\varepsilon}

\def \k {\kappa}
\def \l {\lambda}
\def \n {\nabla}

\def \o {\omega}

\def \D {\Delta}

\def \L {\Lambda}

\def \O {\Omega}

\def \bk {{\bf k}}
\def \bl {{\bf l}}

\def \bu {{\bf u}}

\def \cD {\mathcal{D}}
\def \cE {\mathcal{E}}
\def \cF {\mathcal{F}}
\def \cG {\mathcal{G}}
\def \cH {\mathcal{H}}
\def \cI {\mathcal{I}}

\def \cM {\mathcal{M}}

\def \cP {\mathcal{P}}

\def \cR {\mathcal{R}}

 \def \rme {\mathrm{e}}

\def \const {\mathrm{const}}

\def \one {{\mathds{1}}}

\DeclareMathOperator{\supp}{supp} %

\DeclareMathOperator{\Id}{Id} %
\def \p {\partial}

\def \HI {H\"older inequality}

\def \CK{Csisz\'ar-Kullback inequality}

\def \dx  {\, \mbox{d}x}

\def \dt  {\, \mbox{d}t}

\def \domega  {\, \mbox{d}\omega}
\def \dy  {\, \mbox{d}y}
\def \dz  {\, \mbox{d}z}

\def \dmu  {\, \mbox{d}\mu}

\def \dk  {\, \mbox{d}\kappa}
\def \domega  {\, \mbox{d}\omega}

\def \dtau  {\, \mbox{d}\tau}
\def \drho  {\, \mbox{d}\rho}

\def \dv  {\, \mbox{d} v}

\def \ddt  {\frac{\mbox{d\,\,}}{\mbox{d}t}}

\def \domain {{\O \times \R^n}}

\def \BL {{\mathrm{BL}}}

\def \cMcs {\cM_{\mathrm{CS}}}

\def \cMmt {\cM_{\mathrm{MT}}}

\def \cMfmt {\cM_{\mathrm{\phi}}}

\def \cMseg {\cM_{\mathrm{seg}}}

\def\R{\mathbb{R}}
\def\N{\mathbb{N}}
\def\T{\mathbb{T}}

\def\u{\textbf{u}}
\def\uavg{[\u]_{\rho}}
\def\uavgN{[\u^N]_{\rho^N}}

\def\uavgp{[\u']_{\rho'}}
\def\uavgpp{[\u'']_{\rho''}}

\def\feps{f^{\epsilon}}
\def\emodu{\rme(f^{\epsilon} | \u)}

\def\ueps{\bu^{\epsilon}}
\def\uavgeps{[\bu^{\epsilon}]_{\rho^{\epsilon}}}

\def\rhoeps{\rho^{\epsilon}} 
\def\steps{\st^{\epsilon}}

\def\mueps{\mu^{\epsilon}}
\def\Heps{\cH_{\epsilon}}
\def\Geps{\cG_{\epsilon}}
\def\Ieps{\cI_{\epsilon}}
\def\kappaeps{\kappa^{\epsilon}}
\def\Reps{\cR^{\epsilon}}

\def \st {\mathrm{s}}
\def \wt {\mathrm{w}}

\title[Microscopic, mesoscopic, macroscopic descriptions]
{Microscopic, mesoscopic, and macroscopic descriptions of the Euler alignment system with adaptive communication strength} 

\author[Roman Shvydkoy and Trevor Teolis]{}


\thanks{$^*$Corresponding author: Trevor Teolis}

\begin{document}
\maketitle

\centerline{\scshape
Roman Shvydkoy$^{{\href{mailto:shvydkoy@uic.edu}{\textrm{\Letter}}}1}$
and Trevor Teolis$^{{\href{mailto:tteoli2@uic.edu}{\textrm{\Letter}}}*1}$}

\medskip

{\footnotesize
 \centerline{$^1$University of Illinois at Chicago, United States}
} 

\medskip


\bigskip



\begin{abstract}
The $\st$-model, introduced by the first author in [\textit{Environmental averaging}. EMS Surv. Math. Sci., 11 (2024), no. 2, 277–413], generalizes the Cucker-Smale model by allowing a broader class of velocity averaging protocols, while preserving the 1D conservation law: $\partial_t e + \partial_x (ue) = 0$, $e = \partial_x u + \st$.
  As a continuation of our previous joint work, [\textit{Well-posedness and long time behavior of the Euler Alignment System with adaptive communication strength}, accepted at the Abel Symposium Proceedings, also  arXiv:2310.00269, 2023],
  this paper aims to establish the physical foundations of the $\st$-model by deriving and justifying its microscopic and mesoscopic formulations.
  A distinctive feature of the microscopic system is that it is a discrete-continuous system:
  the position and velocity of the particles are discrete, while $\st$ is an active continuum scalar field. 
  We rigorously derive the mesoscopic description via the mean-field limit and further obtain the macroscopic equations in the monokinetic and Maxwellian regimes.
  Additionally, we analyze the long-time behavior of the kinetic Fokker-Planck-Alignment equation by establishing the relaxation to the Maxwellian in 1D for Favre-averaged velocity fields.
  As a supplement to our previous numerical results, we present particle simulations which reinforce the qualitative similarities between the $\st$-model and the Motsch-Tadmor model. 
\end{abstract}

\tableofcontents



\section{Introduction}
Many models of collective dynamics that describe emergent phenomena incorporate alignment forces as a  driving mechanism of self-organization. 
Notable examples include Reynolds'  classical 3Zone model of flocking \cite{Rey1987}, the popular Viscek model of self-driven particles \cite{VCBCS1995,VZ2012}, Kuramoto synchronization \cite{Kur1975}, and the recently introduced Cucker-Smale alignment model with  metric  communication protocol \cite{CS2007a,CS2007b}. 
The objective of these models is not to precisely replicate physical systems, but rather to capture the core dynamics of emergent phenomena for a large class of communicating agents.  
This leads to a basic question:  
does the model describe a large class of communicating agents 
and, at the same time, remain mathematically tractable in the sense of well-posedness analysis and provability of the emergent phenomena?  

In this paper, we focus on the adaptive strength model introduced in \cite{S-EA}, or $\st$-model for short, which is a variant of the Cucker-Smale model designed to address this question.
The Cucker-Smale model is given by 
\begin{align}
  \label{CS_discrete}
  \begin{cases}
      \dot{x_i} = v_i \\
      \dot{v_i} = \lambda \sum_{j=1}^N m_j \phi(x_i - x_j)(v_j - v_i) .
  \end{cases}
\end{align}
where $x_i$, $v_i$ are the velocity and positions of the agents, $\lambda > 0$ is a scalar that affects the strength of the alignment force, and $\phi$ is a radially decreasing non-degenerate smooth kernel, originally set as $\phi(r) = \frac{1}{(1+ r^2)^{\b/2}}$. 
Unlike previous models where the alignment results were conditioned on the perpetual connectivity of the flock, this model guarantees the unconditional alignment 
\[
\max_{i=1,\ldots,N} |v_i - \bar{v}| \leq C e^{-\d t}, \qquad \max_{i,j=1,\ldots,N} |x_i - x_j| \leq \bar{D},
\]
 for any solution provided that the kernel facilitates sufficiently strong long-range communication between agents:   $\int_0^\infty \phi(r) = \infty$, see \cite{CS2007a,HT2008,HL2009}. In other words, the provability of the emergent phenomena depend only on the parameters of the model. This proved to be useful in many practical situations described in \cite{ABFHKPPS,Darwin,Tadmor-notices,S-book,MMPZ-survey}. Moreover, the result holds under the large crowd limit $N\to \infty$, see \cite{CFRT2010,TT2014}.  
 
Despite its success,  the Cucker-Smale model does not respond well in certain scenarios. 
Motsch and Tadmor argued in  \cite{MT2011} that in heterogeneous formations when, say, a small sub-flock separates itself from a distant large flock, its internal forces become annihilated by the latter if subjected to the Cucker-Smale communication protocol creating unrealistic behavior. A proposed renormalization of the averaging operation
given by 
\begin{align}
  \label{MT_discrete}
  \begin{cases}
      \dot{x_i} = v_i \\
      \dot{v_i} = \frac{\lambda}{\sum_{j=1}^N m_j \phi(|x_i - x_j|)} \sum_{j=1}^N m_j \phi(x_i - x_j)(v_j - v_i) . 
  \end{cases}
\end{align}
restored the balance of forces and similar alignment results  were proved for the new system in \cite{MT2014}.  
Such a renormalization, however, makes the alignment force asymmetric and more singular, which is the reason for the lack of a coherent well-posedness theory at the moment. 

Both models \eqref{CS_discrete} and \eqref{MT_discrete} have a communication strength expressed by a coefficient in front of the alignment force as a prescribed function of the  density of the flock. 
In contrast, the $\st$-model introduces an adaptive strength which evolves in time according to its own transport equation.  
This preserves the regularity properties of the Cucker-Smale model while improving performance in heterogeneous formations.
 However, the lack of control over kinematic properties of the alignment force makes the flocking analysis more challenging.  
 Nevertheless, many classical long-time behavior and well-posedness results of the Cucker-Smale model extend to the $\st$-model, see \cite{ST2024} for a study at the macroscopic level of description.   

This paper aims to (1) introduce the corresponding microscopic and mesoscopic counterparts; (2) to 
establish a rigorous passage between the levels of description; 
and (3) to establish relaxation to a thermodynamic state in one dimension when the velocity averaging itself is given by the Favre averaging such as in  \eqref{CS_discrete}-\eqref{MT_discrete}. 
Before we turn to the technical description of the results, let us present a more detailed motivation of the $\st$-model.

\subsection{Environmental averaging and the $\st$-model} At the macroscopic level of description, the models \eqref{CS_discrete} and \eqref{MT_discrete} above give rise (in the pressureless regime) to a class of Euler alignment systems
\begin{align}
  \label{CS_macroscopic_pressureless}
  \begin{cases}
      \partial_t \rho + \nabla \cdot (\u \rho) = 0 \\
      \partial_t \u + \u \cdot \nabla \u =  \st_{\rho} (\uavg - \u)
        \end{cases}
\end{align}
where $\rho$, $\u$ are the macroscopic density and velocity of the flock on the periodic environment $\T^n$,  $\st_{\rho}$ is a specific communication strength, and $\uavg $ is a density dependent averaging operation, see \cite{FK2019,S-book,S-hypo,CCh2021} for justifications of the macroscopic limit. For example, the CS-model has
\begin{equation}\label{CS}
\st_\rho = \rho \ast \phi, \qquad \ave{u}_\rho = \frac{(u \rho)\ast \phi}{\rho\ast \phi}, \tag{$\cMcs$}
\end{equation}
while the MT-model has
\begin{equation}\label{MT}
\st_\rho = 1, \qquad \ave{u}_\rho = \frac{(u \rho)\ast \phi}{\rho\ast \phi}.\tag{$\cMmt$}
\end{equation}
Many other alignment models can be expressed similarly. For example, the over-mollified version of \ref{MT} introduced in \cite{S-hypo}, is given by
\begin{equation}\label{Mf}
\st_\rho = 1, \qquad  \ave{u}_\rho = \left( \frac{(u \rho)\ast \phi}{\rho\ast \phi} \right)\ast \phi , \tag{$\cMfmt$}
\end{equation}
and the segregation model, based on a partition of unity $\sum_{l=1}^L  g_l = 1$, is given by
\begin{equation}\label{Mseg}
\st_\rho = 1, \quad \ave{u}_\rho(x) = \sum_{l=1}^L g_l(x) \frac{\int_\O u g_l \rho \dy}{\int_\O g_l \rho \dy} . \tag{$\cMseg$}
\end{equation}

The dynamics, energy inequality, momentum conservation, and other physical properties of the system \eqref{CS_macroscopic_pressureless} are largely dictated by functional properties of the chosen family of pairs
\begin{equation*}
\cM = \{ (\st_\rho, \ave{\cdot}_\rho): \rho \in \cP\},
\end{equation*}
where $\cP$ is the space of probability measures on the torus. For example, conservation of momentum holds if 
\[
\int_{\T^n} \uavg \st_\rho \drho = \int_{\T^n} \u \st_\rho \drho
\]
for all $\u$ and $\rho$.  Under certain natural boundedness properties of the pairs, the family $\cM$ is called an {\em environmental averaging} model, see \cite{S-EA}.

The well-posedness theory for hydrodynamic systems, especially in multiple dimensions, is far from being settled, even for the original \ref{CS}-model.
The most satisfactory result was obtained by Carillo, Choi, Tadmor, and Tan \cite{CCTT2016} by noticing that in 1D the Cucker-Smale-based system possesses an extra conservation law:   
\begin{equation}\label{e:e}
  e = \partial_x u + \rho \ast {\phi}, \hspace{5mm} \partial_t e + \partial_x(ue) = 0.
\end{equation}
Thus, the regularity criterion in 1D is stated in terms of a threshold condition on the initial $e_0$: the solution remains smooth iff $e_0 \geq 0$. 
In fact, it is not only useful for global regularity.  It is also instrumental in various studies on long time behavior, aggregation phenomena, estimating disorder of the limiting distribution \cite{LS2019,ST2,LLST-mass}, and extensions to a class of uni-directional flows in multiple dimensions \cite{LS-uni1,LS-uni2}. 

Unfortunately, none of the other listed models, \ref{MT} in particular, satisfy the extra conservation law \eqref{e:e}.
Therefore, even in 1D the question of regularity for such models remains incomplete.

It was observed in \cite{S-EA} that the law \eqref{e:e}, with $e = \p_x u + \st_\rho$, is  equivalent to the continuity equation on the strength
\begin{equation}\label{e:scont}
  \partial_t \st_\rho + \partial_x (\st_\rho \ave{u}_\rho) = 0,
\end{equation}
which only holds, indeed, for the \ref{CS}-model. The idea behind the
 $\st$-model is to depart from prescribing strength as a function of $\rho$. Instead, let the strength evolve according to its own  continuity equation \eqref{e:scont}, thus, adding another unknown to the system:
 \begin{align}
  \label{SM_macroscopic_pressureless}
  \tag{SM}
  \begin{cases}
      \partial_t \rho + \nabla \cdot (\u \rho) = 0 \\
      \partial_t \st + \nabla \cdot (\st \uavg) = 0, \qquad \st \geq 0 \\ 
      \partial_t \u + \u \cdot \nabla \u = \st(\uavg - \u).
  \end{cases}
\end{align}
The new system, by design, admits the $e$-quantity in 1D
\begin{equation}\label{e:eslaw}
  e = \partial_x u + \st, \hspace{5mm} \partial_t e + \partial_x(ue) = 0.
\end{equation}

\begin{remark}
When $\ave{\cdot}_\rho$ is given by the Favre filtration, the natural choices of the initial strength are given by $\st_0 = \rho_0 \ast \phi$ and $\st_0 = 1$. 
If $\st_0 = \rho_0 \ast \phi$, we recover the Cucker-Smale system. If $\st_0 = 1$, the model reproduces  the Motsch-Tadmor protocol at the onset, and then naturally deviates from it following its own transport equation.  
\end{remark}
 
 Despite the departure from \ref{MT}, the asymptotic behavior of the $\st$-model is qualitatively similar to that of the \ref{MT}-model in heterogeneous formations. 
 The numerical evidence based on particle simulations is provided in Section \ref{sec:numerics} (see also \cite{ST2024} for numerical evidence at the macroscopic level).  At the same time, thanks to the reclaimed conservation law \eqref{e:eslaw} the $\st$-model has similar analytical properties to \ref{CS}, which \ref{MT} lacks.  We recall some of these properties from our previous study \cite{ST2024}. 
 
 The study was performed for a particular form of the strength and for models with  velocity averaging given by the Favre filtration, which was introduced in the context
of non-homogeneous turbulence in \cite{Favre},
 \begin{align}
 \label{Favre_averaging}
  	\uavg = \u_F = \frac{(u \rho)\ast \phi}{\rho\ast \phi},
 \end{align}
  like in the models of interest, \ref{CS}, \ref{MT}. We represented $\st$ as a weighted variant of the \ref{CS}-strength, 
\begin{equation}\label{e:sw}
 \st =(\rho\ast {\phi})  \wt .
\end{equation}
Since $\rho\ast {\phi}$ satisfies the same continuity equation as $\st$, the weight $\wt$ satisfies the pure transport equation, and thus the system becomes
\begin{align}
  \label{eqn:WM_macroscopic}
  \tag{WM}
  \begin{cases}
      \partial_t \rho + \nabla \cdot (\u \rho) = 0 \\
      \partial_t \wt + \u_F \cdot \nabla \wt = 0, \qquad \wt \geq 0  \\
      \partial_t \u + \u \cdot \nabla \u = \wt \big( (\bu \rho) \ast \phi - \u (\rho \ast \phi) \big).
  \end{cases}
\end{align}
We call this variant of the $\st$-model, the $\wt$-model.
Here is a brief summary of results obtained for the $\wt$-model in \cite{ST2024}:
\begin{enumerate} [(a)]
  \item Local well-posedness and global well-posedness for small data in multi-D;
  \item Global well-posedness in 1D and for unidirectional flocks  under the threshold $e \geq 0$;
  \item Cucker-Smale $L^{\infty}$-based alignment; conditional $L^2$-based alignment;
  \item Strong flocking in 1D (i.e. convergence of the density to a limiting distribution);
  \item Estimates on the limiting distribution of the flock in 1D.
\end{enumerate}

\subsection{Microscopic and mesoscopic levels of description} The study of \cite{ST2024} was done  at the level of the hydrodynamic description,  while the microscopic and mesoscopic levels have remained unattended and even undefined. For the classical \ref{CS}-model the mean-field limit has been established in \cite{HL2009}, and for more general models in \cite{S-book,S-EA}. Passing from the kinetic to the hydrodynamic level has been described in the monokinetic \cite{FK2019,CCh2021} and Maxwellian \cite{KMT2015} regimes for the Cucker-Smale model, and in \cite{S-EA} for more general models.

 The main issue that prevents the application of the classical recipe for the $\st$-model is that the system \eqref{SM_macroscopic_pressureless} has two distinct  characteristic flows, along $\u$ and $\uavg$, which need to be properly incorporated in the particle dynamics. In this paper, we propose to circumvent this problem by considering a hybrid version with discrete/continuous dynamics. 

In order to properly formulate these lower levels of description, we assume that the averaging operator has an integral representation and against a smooth kernel $\Phi_{\rho}$:  
\begin{align}
  \label{uavg_integral_representation}
  \uavg = \int \Phi_{\rho}(x,y) \u(y) \rho(y) \dy, \hspace{5mm} \Phi_{\rho} \geq 0 \text{ is smooth}.
\end{align}
We also assume that $\Phi_{\rho}$ is right stochastic:
\begin{align}
    \label{eqn:kernel_right_stochastic}
    \int_{\T^n} \Phi_{\rho}(x,y) \rho(y) \dy = 1
\end{align}
in order for the averaging to preserve constants and for it to be an order-preserving map: $L^\infty \to L^\infty$.  We refer to \eqref{reg1_kernel} and \eqref{reg2_kernel} for a more precise list of assumptions on the kernel, and to 
 Table~\ref{t:kernels} for the list of the reproducing kernels of our core models.  It should be noted that the needed regularity of the Favre-based model is satisfied only for non-degenerate communication kernels, 
 \begin{equation}\label{e:nondeg}
\inf_{\T^n} \phi > 0.
\end{equation}

\begin{table}
\begin{center}
\caption{Reproducing kernels}\label{t:kernels}
\begin{tabular}{  c | c | c | c  } 
 MODEL &     \ref{CS}  \&  \ref{MT}  &  \ref{Mf}  &  \ref{Mseg}  \\
  \hline
$\Phi_\rho$ &   $\displaystyle{ \frac{\phi(x-y)}{\rho\ast \phi(x)} }$ &$\displaystyle{\int_{\T^n} \frac{\phi(x-z) \phi(y-z)}{\rho\ast \phi(z)} \dz}$  &  $\displaystyle{\sum_{l=1}^L \frac{g_l(x) g_l(y)}{\int_{\T^n} \rho g_l \dy}}$  
 \end{tabular}
\end{center}
\end{table}
Under such assumptions, the velocity averaging operation yields a smooth field even with empirical data:
\begin{equation*}
	\begin{split}
	&\uavgN(x) =  \sum_{j=1}^N m_j \Phi_{\rho^N}(x, x_j) v_j \in C^\infty(\T^n), \\
	& \qquad   \rho^N = \sum_{j=1}^N m_j \delta_{x_j}, 
	 \quad \u^N = \sum_{j=1}^N v_j \one_{x_j}.
	\end{split}
\end{equation*}

The microscopic description of \eqref{SM_macroscopic_pressureless} is then given by an ODE system describing the position and velocity of particles coupled 
with a PDE describing the transport of the strength function:
\begin{align}
    \label{SM_microscopic}
    \begin{cases}  
        \dot{x_i} = v_i \\
        \dot{v_i} = \lambda \st(x_i) (\uavgN(x_i) - v_i ),  \\
        \partial_t \st + \nabla_x \cdot (\st \uavgN) = 0 .
    \end{cases}
\end{align}
The existence and uniqueness of solutions to \eqref{SM_microscopic} will be proved in \sect{sec:well-posedness_of_sm_microscopic}.
\begin{theorem}\label{t:gwpmicro}
The system \eqref{SM_microscopic} admits unique global solution 
\[
(\{x_i\}_{i=1}^N, \{v_i\}_{i=1}^N, \st) \in C([0,\infty); \R^{nN} \times  \R^{nN} \times H^k(\T^n)), \quad  k > n/2 + 2,
\]
from any data in the same class.
\end{theorem}

The corresponding kinetic model is given by a Vlasov-alignment equation coupled with the same transport of the strength function along the averaged macroscopic velocity:
\begin{align}
    \label{SM_kinetic}
    \begin{cases}
        \partial_t f + v \cdot \nabla_x f + \lambda \nabla_v \cdot (\st (v - \uavg) f) = 0 \\
        \partial_t \st + \nabla_x \cdot (\st \uavg) = 0 .
    \end{cases}
\end{align} 
Here, $\rho$ and $\rho \u$ are the macroscopic variables given by: 
\begin{align}
  \label{defn:density_momentum}
  \rho = \int_{\R^n} f \dv, \hspace{5mm} \rho \u = \int_{\R^n} v f \dv . 
\end{align}

A passage from \eqref{SM_microscopic} to \eqref{SM_kinetic} is facilitated through the mean field limit, which 
 in present settings consists of establishing the weak convergence of the empirical measures
\begin{equation*}
    \mu^N_t = \sum_{i=1}^N m_i \delta_{x_i(t)} \otimes \delta_{v_i(t)}  \to \mu_t
\end{equation*}
where $(x_i(t), v_i(t))_{i=1}^N$ solve the first two equations in \eqref{SM_microscopic}, as well as 
\[
\st^N \to \st
\]
in a stronger sense.  The stronger convergence of the strength is available due to the inherited regularity of the strength functions from the kernel.
More precisely,  the corresponding mean field limit theorem will be proven in Section \ref{sec:mean_field_limit}.  We denote by $W_1$ the classical Wasserstein-1 metric, see \sect{s:notation} for the definition.
\begin{theorem} 
  \label{thm:mean_field_limit}
  Suppose that the velocity averaging $\uavg$ has the integral representation \eqref{uavg_integral_representation} with 
  the kernel satisfying \eqref{reg1_kernel} and \eqref{reg2_kernel}. 
  Given $T>0$, $R>0$, $k \geq 0$,  $\mu_0 \in \cP(\T^n \times B_R)$, and $\st_0 \in C^{\infty}(\T^n)$, then there exists a unique 
  weak solution $(\mu, \st) \in C_{w^*}([0,T]; \cP(\T^n \times B_R)) \times C([0,T]; C^k(\T^n))$ to \eqref{SM_kinetic}.  
  Moreover, the solution can be obtained as follows:  if $\st_0^N = \st_0$ and the empirical measures $\mu_0^N$ are constructed from agents $(x_i^0,v_i^0) \in \T^n \times \R^n$ satisfying $W_1(\mu^N_0, \mu_0) \to 0$, then 
  \begin{enumerate} [(i)]
    \item $\sup_{t \in [0,T]} W_1(\mu^N_t, \mu_t) \to 0$, and 
    \item $\sup_{t \in [0,T]} \|\st_t^N - \st_t\|_{C^k} \to 0$ for any $k \geq 0$ where $\st^N$ and $\st$  solve the corresponding transport equations in \eqref{SM_microscopic}  and \eqref{SM_kinetic}, respectively.
  \end{enumerate}
\end{theorem}

Taking moments of the $\st$-Vlasov-Alignment equation \eqref{SM_kinetic}, we obtain the following hydrodynamic system
\begin{align}
  \label{SM_vmoments}
  \begin{cases}
      \partial_t \rho + \nabla \cdot (\u \rho) = 0 \\
      \partial_t \st + \nabla \cdot (\st \uavg) = 0 \\ 
      \partial_t (\rho \u) + \nabla \cdot (\rho \u \otimes \u) + \nabla \cdot \cR = \rho \st(\uavg - \u), 
  \end{cases}
\end{align}
where $\cR$ is the Reynolds stress tensor 
\begin{align}
  \label{defn:reynolds_stress_monokinetic}
  \cR(t,x) = \int_{\R^n} (v - \u(t,x)) \otimes (v - \u(t,x)) f(t,x,v) \dv . 
\end{align}

Here we encounter the classical closure problem: the Reynolds stress still depends upon $f$. The classical way to enforce $\cR$ to close on the macroscopic variables is to consider mean free path limits $\e \to 0$ of solutions to the forced kinetic model enhanced with $\frac{1}{\e} F(f^\e)$,  where $F$ is either a pure local alignment  or a Fokker-Planck-alignment force. In the former case it drives the solution to monokinetic distribution $f^\e \to  \rho(t,x) \delta(v - \u(t,x)) $  and therefore the Reynolds stress term vanishes, $\cR = 0$, resulting in the pressureless Euler-alignment system. In the latter case, this leads to the local Maxwellian distribution $f^\e \to \frac{\rho(t,x)}{(2\pi)^{n/2}} e^{-\frac{|v - \u(t,x)|^2}{2}}$, and $\cR = \rho \Id$ resulting in the appearance of the isothermal pressure  $p = \rho$.

We prove adaptations of these limits to the $\st$-model. 
\begin{theorem} (Monokinetic limit)
	The macroscopic quantities $\rho$ and $\u \rho$ for solutions to the kinetic $\st$-model with pure local alignment force converge to solutions to the pressureless macroscopic system \eqref{SM_macroscopic_pressureless}.
\end{theorem}
\begin{theorem} (Maxwellian limit)
	When $\uavg = \u_F$ (i.e. the $\wt$-model), the macroscopic quantities $\rho$ and $\u \rho$ for solutions to the kinetic $\wt$-model with strong Fokker-Planck-alignment force converge to solutions to the macroscopic system with isothermal pressure \eqref{SM_hydrodynamic_maxwellian}.
\end{theorem}

The precise statements are given \thm{thm:monokinetic_limit} and  \thm{thm:maxwellian_limit}.
In general terms, these results are similar to those obtained for \ref{CS}-model in \cite{FK2019,KMT2015}. However, the incorporation of the adaptive strength requires more scrutiny.  For technical reasons, the Maxwellian limit is proved unconditionally in the particular case of the $\wt$-model only. To treat the general $\st$-model, solutions must satisfy a uniform bound on the strength-functions \eqref{e:seuniform}-- notably, this is the only assumption which separates us from a result for the general $\st$-model.

We also prove a conditional relaxation result in 1D.  The precise statement is given in Theorem \ref{thm:relaxation_wm}.
\begin{theorem} (Relaxation to the Maxwellian in 1D)
	When $\uavg = \u_F$ (i.e. the $\wt$-model), $n=1$, and the initial variation of $\wt$ is small, solutions to the kinetic Fokker-Planck-Alignment $\wt$-model \eqref{e:FPA} relax exponentially fast to the Maxwellian.
\end{theorem}

\subsection{Numerical evidence for similar qualitative behavior to Motsch-Tadmor}
\label{sec:numerics}
To make the case for physical relevance of the $\st$-model, 
we present numerical evidence, at the microscopic level, that the $\wt$-model with $\wt_0 = 1/(\rho_0 \ast {\phi})$, and hence the $\st$-model,
displays similar qualitative behavior to that of Motsch-Tadmor in heterogeneous formations.  
We first clarify the qualitative behavior that we are seeking. 

Following  \cite{MT2011} we let  $m$, $M$ with $M >\!\!> 1 \approx m$ be the masses of a separated small and large clusters in a flock, respectively.  
Then for any agent $i$ in the small cluster under the Cucker-Smale protocol, we expect to see a stalled dynamics
\[
\dot{v}_i \approx 0,
\]
while the large cluster evolves naturally.  Under the Motsch-Tadmor protocol we expect the dynamics of the small cluster to be dominated by its own local communication
\[
\dot{v}_i \approx \frac{\lambda}{\sum_{i' \in I}^N m_{i'} \phi(|x_i - x_{i'}|)} \sum_{i' \in I} m_{i'} \phi(|x_i - x_{i'}|) (v_{i'} - v_i)
\]
and since the local strength in this case $\sum_{i' \in I}^N m_{i'} \phi(|x_i - x_{i'}|) \sim 1$, we expect to observe the same Cucker-Smale dynamics as in a homogeneous formation.  We refer to   \cite{MT2011} for the detailed scaling computation.

In our $\wt$-model simulation, we are looking for the same  qualitative behavior as under the Motsch-Tadmor protocol. 
For comparison, we also show the Cucker-Smale simulation. 
We aim to see the following.
\begin{enumerate} 
  \label{numerics:expected_behavior}
  \item [(Q$_{cs}$)] For the Cucker-Smale model:  The small flock proceeds linearly as if there were no force on it. 
  \item [(Q$_{\wt}$)] For the $\wt$-model with $\wt_0 = 1/(\rho_0 \ast {\phi})$:  The small flock behaves according to Cucker-Smale, but independently of the large flock. 
  The velocities of the small flock will therefore align to the average velocity of the small flock. 
\end{enumerate}

The computation below is performed on the discrete $\wt$-model, a special case of \eqref{SM_microscopic}, which reads
\begin{equation}
  \label{WM_discrete}
  \begin{cases}
      \dot{x_i} = v_i \\
      \dot{v_i} = \lambda \wt(x_i) \sum_{j=1}^N m_j \phi(|x_i - x_j|) (v_j - v_i) \\
      \partial_t \wt + \uavgN \cdot \nabla_x \wt = 0 . 
  \end{cases}
\end{equation}

The solutions are computed on the 2D unit torus, $\T^2$.  
We consider initial data consisting of two clusters as described above.
The parameters of the experiment are as follows. 
\begin{itemize}
  \item The scalar strength of the alignment force is $\lambda = 10$.
  \item $\phi(r) = \frac{1}{(1 + r^2)^{80/2}}$.
  \item $\rho_0^N$ is identical in the Cucker-Smale and the $\wt$-model simulation.  It is shown in Figure \ref{plots:solns_CS} (and Figure \ref{plots:solns_WM_with_MT_data}) as the leftmost picture. 
\end{itemize}
The kernel is periodized so that the distance $r$ measures the distance on $\T^2$. 
The heavy cluster is indicated by black particles and 
the light cluster is indicated by white particles. 
Each black particle has about $100$ times the mass of a white particle. 

\begin{figure}[!h]

  \begin{center}
          \begin{tabular}{ccc}
              \includegraphics[width=0.3\linewidth]{\detokenize{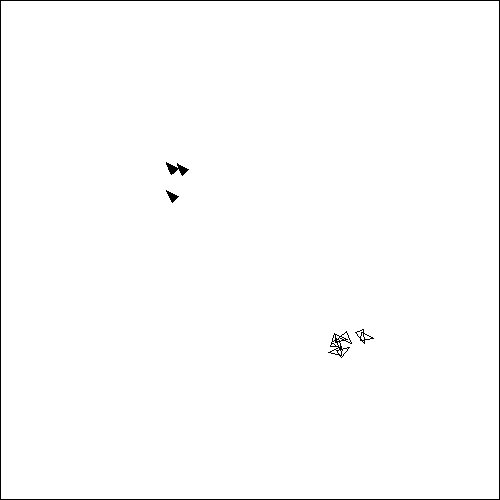}} &
              \includegraphics[width=0.3\linewidth]{\detokenize{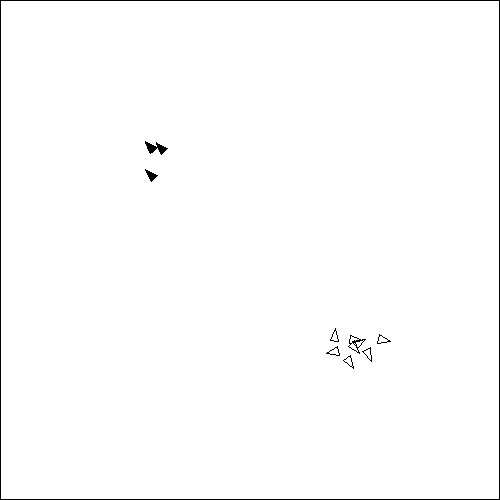}} &
              \includegraphics[width=0.3\linewidth]{\detokenize{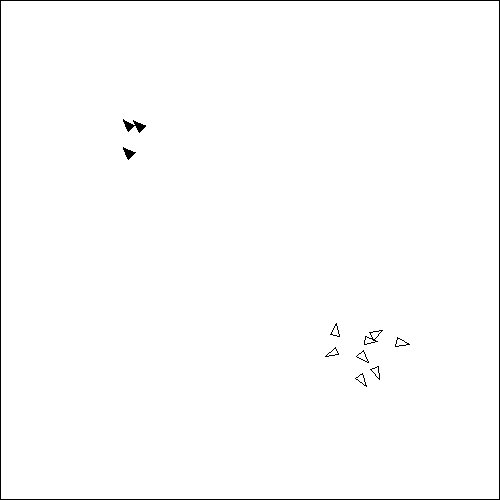}} 
          \end{tabular}
  \end{center} 
  \caption{The computed solution of the \textbf{Cucker-Smale} model at three different time steps. 
  The leftmost image is the initial configuration of the flock and time moves left to right.  The dynamics of the light (white) flock conforms with observation (Q$_{cs}$).
  \label{plots:solns_CS}
  }
\end{figure}

\begin{figure}[!h]
  \begin{center}
          \begin{tabular}{ccc}
              \includegraphics[width=0.3\linewidth]{\detokenize{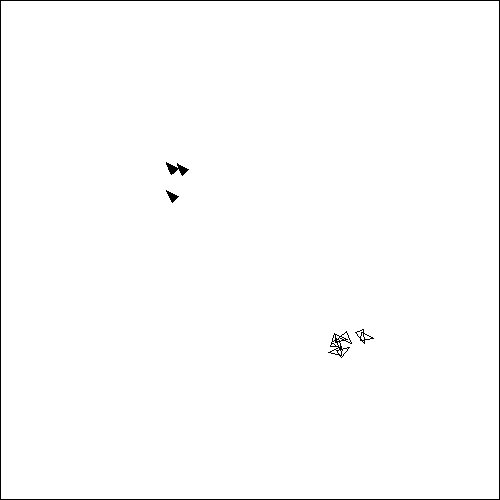}} &
              \includegraphics[width=0.3\linewidth]{\detokenize{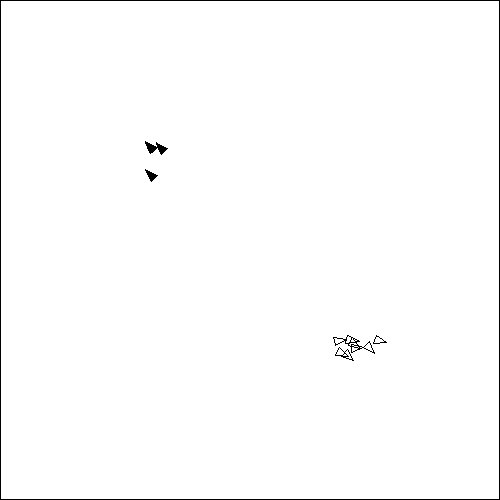}} &
              \includegraphics[width=0.3\linewidth]{\detokenize{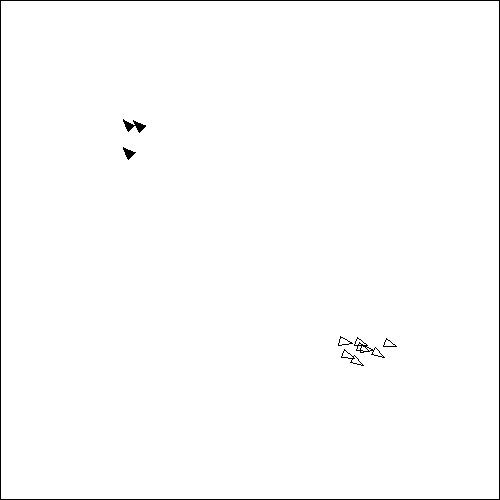}} 
          \end{tabular}
  \end{center} 
  \caption{The computed solution of the \textbf{$\wt$-model with Motsch-Tadmor initial data}, i.e. $\wt_0 = 1/(\rho \ast \phi)$, at three different time steps. 
  The leftmost image is the initial configuration of the flock and time moves left to right.  The dynamics of the light (white) flock conforms with observation (Q$_{\wt}$).
  \label{plots:solns_WM_with_MT_data}
  }
\end{figure}

The simulations shown in Figures \ref{plots:solns_CS} and \ref{plots:solns_WM_with_MT_data} show that the small agents in the Cucker-Smale case proceed linearly while the small agents in $\wt$-model case align to the average velocity of the small flock.  
This is the desired qualitative behavior.

\subsection{Outline}
\label{sec:outline}

The rest of the paper will be organized as follows. 
In Section \ref{sec:inherited_reg_from_the_kernel}, we will show that the velocity averaging and strength inherit the regularity of the kernel. 
In preparation for the mean field limit, the well-posedness of the microscopic $\st$-model \eqref{SM_microscopic} is established in Section \ref{sec:well-posedness_of_sm_microscopic}. 
The mean field limit is proved in Section \ref{sec:mean_field_limit}.  The hydrodynamic limits are proved in Section \ref{sec:hydrodynamic_limits}. 
Finally, in Section \ref{sec:relaxation}, we establish the relaxation to the Maxwellian in 1D for the mesoscopic $\wt$-model \eqref{e:FPA} provided the variation of the weight is small.

\subsection{Assumptions and Notation} \label{s:notation}

Our results will be stated for the torus $\T^n$.  
Letting $\rho, \rho', \rho'' \in \cP(\T^n)$, we assume throughout the paper, unless stated otherwise, that $\uavg$ has the integral representation $\eqref{uavg_integral_representation}$ 
and that its reproducing kernel $\Phi_{\rho}(x,y) \geq 0$ satisfies the following uniform regularity assumptions: 
\begin{align}
    \label{reg1_kernel}
    \tag{$\Phi$Reg1} &\| \partial_{x,y}^k \Phi_{\rho}\|_{\infty} \leq C_k \\
    \label{reg2_kernel}
    \tag{$\Phi$Reg2} &\| \partial_{x,y}^k (\Phi_{\rho'}- \Phi_{\rho''}) \|_{\infty} \leq C_k W_1(\rho', \rho'').
\end{align}
As previously noted, for the Favre-based models these assumptions are trivially satisfied if the defining communication kernel is non-degenerate \eqref{e:nondeg}.

$C^k$ is the space of $k$ continuously differentiable functions with the usual norm $\|f\|_{C^k} = \sum_{i=0}^k \|f\|_{C^i}$. We will use $\odot$ to denote component-wise multiplication of vectors (i.e. $a \odot b = (a_1 b_1, a_2 b_2, \dots)$).
Subscripts $-$ and $+$ will be used as a shorthand for infima and suprema. For instance, $f_- = \inf_{x \in \T^n} f(x)$, $f_+ = \sup_{x \in \T^n} f(x)$. 
We will use $(f_1, f_2) = \int_{\T^n} f_1 f_2 \dx$ to denote the $L^2$ inner product and $(f_1, f_2)_{h} = \int_{\T^n} f_1 f_2 h \dx$ to denote the 
weighted $L^2$ inner product.  Per  \cite{S-EA}, we will use the notation $\kappa_{\rho} = \st \rho$ for the kinematic strength measure.
For instance, $(\cdot, \cdot)_{\kappa_{\rho}}$ denotes the $L^2$ inner product with respect to the measure $\kappa_{\rho}$.  We also denote 
\[
L^p(\rho) = \{ f \in \cD': \int_{\T^n} |f|^p \rho \dx <\infty\}.
\]

We will use the dual definition of the Wasserstein-1 metric on the space $\cP(\O)$, where $\O$ is a Borel space:
\begin{align*}
  W_1(\mu, \nu) = \sup_{\| \n g \|_\infty \leq 1} \big| \int_{\O} g(\o) (d\mu(\o) - d\nu(\o)) \big| ,
\end{align*}
and the Bounded Lipschitz metric which can be applied to an arbitrary pair of signed measures
\begin{align*}
  W_{\BL}(\mu, \nu) = \sup_{\|g\|_\infty, \| \n g \|_\infty \leq 1} \big| \int_{\O} g(\o) (d\mu(\o) - d\nu(\o)) \big| .
\end{align*}
We also use the classical definition of the Wasserstein-2 metric: 
\begin{equation*}
    W_2^2(\mu, \nu)= \inf_{\gamma \in \Pi(\mu, \nu)} \int_{\O \times \O} |w_1 - w_2|^2 \mbox{d}\gamma(w_1, w_2) ,
\end{equation*}
where $\Pi(\mu, \nu)$ is the set of measures with marginals $\mu$ and $\nu$, see \cite{Villani-optimal}.

\section{Inherited Regularity from the Kernel}\label{sec:inherited_reg_from_the_kernel}

Our  assumptions on the regularity of the kernel naturally translate into regularity of the averaging and the strength-function. This will be reflected in the following two inheritance lemmas.

\begin{lemma} (Inherited regularity of the velocity averaging)
  \label{prop:regularity_velocity_averaging}
  Suppose that the kernel $\Phi_\rho$ satisfies the uniform regularity assumptions \eqref{reg1_kernel} and \eqref{reg2_kernel}. 
  Let $\rho, \rho', \rho''  \in \cP$.  
  Then for all $k \geq 0$  we have
  \begin{align} 
    \label{reg1_uavg}
    \tag{$\u$Reg1} &\| \uavg\|_{C^k}  \leq C \|\u\|_{L^1(\rho)}, \\
    \label{reg2_uavg}
    \tag{$\u$Reg2} & \|\uavgp - \uavgpp\|_{C^k}  \leq C \|\u'\|_{L^1(\rho')} W_\BL(\rho', \rho'') +  C W_\BL(\u' \rho', \u'' \rho'').
  \end{align}
\end{lemma}

\begin{proof}
  For \eqref{reg1_uavg}, place all of the derivatives on the kernel and use \eqref{reg1_kernel}.
  For \eqref{reg2_uavg}, we write 
  \begin{align*}
    \uavgp - \uavgpp 
      &= \int_{\T^n} \Phi_{\rho'}(x,y) \u'(y) \rho'(y) \dy - \int_{\T^n} \Phi_{\rho''}(x,y) \u''(y) \rho''(y) \dy \\
      &= \int_{\T^n} (\Phi_{\rho'}(x,y) - \Phi_{\rho''}(x,y)) (\u' \rho')(y) \dy  \\
      &\quad+  \int_{\T^n} \Phi_{\rho''}(x,y) ( (\u'\rho')(y) - (\u''\rho'')(y)) \dy . 
  \end{align*}
  Once again, placing the derivatives on the kernel and using \eqref{reg1_kernel} and \eqref{reg2_kernel} we arrive at \eqref{reg2_uavg}. 
\end{proof}

The strength subsequently inherits regularity from the velocity averaging. 
Let us denote 
  \[
    J := \sup_{t\in [0,T]} \| \u \|_{L^1(\rho)}.
  \]
 We should remark that  $J$ is controlled by, for example the energy, or the maximum of $\u$, which holds a priori for the microscopic and kinetic $\st$-model, \eqref{SM_microscopic} and \eqref{SM_kinetic}, 
  due to the maximum principle on the velocity.  
  The maximum principle also holds for the mesoscopic $\st$-model with the strong local alignment force for the monokinetic limiting regime, \eqref{SM_kinetic_monokinetic}
  (the justification is provided in Section \ref{sec:monokinetic_regime}). 
  However, for the kinetic $\st$-model with strong Fokker-Planck penalization force, \eqref{SM_kinetic_maxwell_la}, there is no control on $J$ except when $\st$ is written in the $\wt$-form \eqref{e:sw}.
 This is part of the reason why the Maxwellian limit is proved only for the $\wt$-model.

\begin{lemma} (Inherited regularity of the strength)
    \label{prop:regularity_strength}
    Suppose that the kernel $\Phi_\rho$ satisfies the uniform regularity assumptions \eqref{reg1_kernel} and \eqref{reg2_kernel}.
  Let $\st$ satisfy the continuity equation
  \[
   \partial_t \st + \nabla_x \cdot (\st \uavg)  = 0.
   \]
   Then for all $k \geq 0$, we have
   \begin{equation}\label{}
 \label{reg1_st}
      \tag{$\st$Reg1}  \sup_{t \in [0,T]}  \|\st\|_{C^k} \leq C(k,J).
\end{equation}
If two strength-functions $\st', \st'' \in C([0,T]; C^{k+1}(\T^n))$  solve their respective continuity equations, then for all $k \geq 0$:
  \begin{equation} \label{reg2_st}
   \tag{$\st$Reg2}
  \begin{split}
       \sup_{t \in [0,T]} \|\st' - \st''\|_{C^k}  &\leq C \|\st_0' - \st_0''\|_{C^k}  \\
      & \quad +  C \int_0^T [ W_\BL(\rho', \rho'') +  W_\BL(\u' \rho', \u'' \rho'') ] \dt,
      \end{split}
\end{equation}
    where $C := C(k, J',J'')$. 
    \begin{proof}
    Let us apply $k$ derivatives of the continuity equation and evaluate at the maximum. We obtain, by \eqref{reg1_uavg},
    \[
    \ddt  \|\st\|_{C^k} \leq \|\uavg\|_{C^{k+1}} \|\st\|_{C^k}  \lesssim \|\u\|_{L^1(\rho)} \|\st\|_{C^k} .
    \]
    Hence, \eqref{reg1_st} follows.
    
    To prove \eqref{reg2_st}, apply the $k^{th}$  derivative for the difference of equations, and use \eqref{reg1_uavg}-\eqref{reg2_uavg} to get 
\begin{equation*}\label{}
\begin{split}
            \partial_t \|\st' - \st''\|_{C^k} & \lesssim  \|\uavgp - \uavgpp\|_{C^{k+1}} \|\st'\|_{C^{k+1}} + \|\uavgpp\|_{C^{k+1}} \|\st' - \st''\|_{C^k} \\
            & \leq C(J') (W_\BL(\rho', \rho'') +   W_\BL(\u' \rho', \u'' \rho''))+ \|\u''\|_{L^1(\rho'')} \|\st' - \st''\|_{C^k} . 
\end{split}
\end{equation*}
    The bound \eqref{reg2_st} follows from the Gr\"onwall Inequality.
    \end{proof}
  \end{lemma}


\section{Well-posedness of the agent based model}
\label{sec:well-posedness_of_sm_microscopic}

The well-posedness of the discrete-continuous system stated in \thm{t:gwpmicro} must be established before addressing the mean field passage to the kinetic system.  We note that the choice of $k > n/2 + 2$ guarantees that $\|\partial^2 \st\|_{\infty} \lesssim \|\st\|_{H^k}$ for an arbitrary second order partial derivative $\partial^2$ by the Sobolev Embedding Theorem. 

Let us recall that the kernel satisfies the standing regularity assumptions \eqref{reg1_kernel} and \eqref{reg2_kernel}. Thus, from Lemma \ref{prop:regularity_velocity_averaging}, $\uavgN$ inherits the regularity of the kernel and, in particular, satisfies \eqref{reg1_uavg}. 
For the remainder of this section, we will use $C$ to denote a constant depending on $k$ only and may change line by line. 

Consider the following viscous regularization of \eqref{SM_microscopic}:
\begin{equation}
  \label{SM_microscopic_viscous}
  \begin{cases}
      \dot{x_i} = v_i \\
      \dot{v_i} = \lambda \st(x_i) (\uavgN(x_i) - v_i ) \\
      \partial_t \st + \nabla_x \cdot (\st \uavgN) = \epsilon \Delta \st .
  \end{cases}
\end{equation}
For the moment, we avoid writing the explicit dependence of $x_i, v_i, \st$ on $\epsilon$ for the sake of brevity.  They are not to be confused with solutions 
to the unregularized system \eqref{SM_microscopic}.  
Let $X = C([0,T]; \R^{2nN} \times H^k(\T^n))$ and define $Z(t) := (\{x_i(t)\}_{i=1}^N, \{v_i(t)\}_{i=1}^N, \st(t))$. 
The norm $\|\cdot\|_X$ is given by:
\begin{equation*}
  \|Z\|_X = \sup_{t \in [0, T]} \max_{i = 1 \dots N} \|x_i(t)\| + \|v_i(t)\| + \| \st(t,\cdot) \|_{H^k(\R^n)}.
\end{equation*}
Define the map $\cF$ by 
\begin{align}
  \label{eqn:Duhamel_formula}
  \cF( Z(t) ) =   
    \begin{bmatrix}
      1 \\
      1 \\
      e^{\epsilon t \Delta}
  \end{bmatrix} \odot Z_0 
  + \int_0^t  
   \begin{bmatrix}
      1 \\
      1 \\
      e^{\epsilon (t-\tau) \Delta}
  \end{bmatrix} \odot 
  A(\tau) d\tau,
\end{align}
where $A$ represents all of the non-laplacian terms.  Existence and uniqueness of solutions to \eqref{SM_microscopic_viscous} amounts to showing 
that $\cF : B_1(Z_0) \mapsto B_1(Z_0)$ and that it is a contraction mapping, where $B_1(Z_0)$ is the ball of radius $1$ centered at $Z_0$ in $X$.
Contractivity will follow similarly from invariance. 
For invariance, we aim to show that: 
\begin{equation*}
\begin{split}
  \|\cF ( Z(t) ) - Z_0 \|_X &\leq 
      \Big\|    
      \begin{bmatrix}
        1 \\
        1 \\
        e^{\epsilon t \Delta}
    \end{bmatrix}  \odot Z_0 - Z_0 \Big\|_X 
    + \Big\| 
          \int_0^t 
          \begin{bmatrix}
            1 \\
            1 \\
            e^{\epsilon (t-\tau) \Delta}
        \end{bmatrix}  \odot 
          A( Z(\tau) ) d\tau \Big\|_X \\
          &\leq 1 .
 \end{split}
\end{equation*}
The first term is small for small $T$ due to the continuity of the heat semigroup.  
For second term, we will treat each component individually. 
For the $x_i$-component, 
$$\|x_i(t) - x_i(0)\| \leq T \max_i \|v_i(0)\| .$$
For the $v_i$-component, we have 
\begin{align*}
  \|v_i(t) - v_i(0) \| \leq T \|\st\|_{\infty} (\|\uavgN\|_{\infty} + \max_i \|v_i(0)\|) \leq C T (\|Z_0\| + 1) .
\end{align*}
For the $\st$-component, we will use the analyticity property of the heat semigroup,
\begin{align*}
  \|\nabla e^{\epsilon t \Delta} f\|_{L^2} \leq \frac{1}{\sqrt{\epsilon t}} \|f\|_{L^2} ,
\end{align*}
along with the product estimate to get 
\begin{align*}
  &\Big\| \int_0^t e^{\epsilon (t-\tau) \Delta} \nabla_x \cdot (\st \uavgN)  d\tau \Big\|_{H^k} \\
    &\leq \frac{2 T^{1/2}}{\epsilon^{1/2}} \sup_{t \in [0,T]} \|\st \uavgN \|_{H^k} \\
    &\leq \frac{2 T^{1/2}}{\epsilon^{1/2}} \sup_{t \in [0,T]} ( \|\st \|_{\infty} \| \uavgN \|_{H^k} + \|\st \|_{H^k} \| \uavgN \|_{\infty}) \\
    &\leq \frac{2 T^{1/2}}{\epsilon^{1/2}} (\|Z_0\|_X + 1) . 
\end{align*}
For small enough $T$, we have invariance.  Contractivity of $\cF$ follows from similar estimates. 
This time interval of existence $T$ could depend on $\epsilon$. 
To establish that the there is a common time interval of existence independent of $\epsilon$, we establish an $\epsilon$-independent energy estimate.
For the $x_i$ and $v_i$ components, the estimates follow easily from the invariance estimates.  We record them here. 
$$\|x_i(t)\| \leq \|x_i(0)\| + t \max_i \|v_i(0)\|,$$  
and
\begin{align*}
  \|v_i(t)\| \leq \|v_i(0)\| + t \|\st\|_{\infty} (\|\uavgN\|_{\infty} + \max_i \|v_i(0)\|) \leq C( 1 + t \|\st\|_{H^k} ).
\end{align*}
For the strength, we multiply \eqref{SM_microscopic_viscous} by $\partial^{2j} \st$ and integrate by parts. We obtain for any $0 \leq j \leq k$, 
\begin{equation*}
\begin{split}
  \label{eqn:st_Hj_norm}
  \frac{d}{dt} \|\st\|_{\dot{H}^j} 
    &= \int \nabla_x \cdot \uavgN |\partial^j \st|^2 \dx \\
    & \quad - \int (\partial^j (\uavgN \cdot \nabla_x \st) - \uavgN \cdot \nabla_x \partial^j \st) \partial^j \st \dx \\
    & \quad - \int \partial^j ((\nabla_x \cdot \uavgN) \st ) \partial^j \st \dx - \epsilon \int |\partial^j \nabla \st|^2 \dx . 
    \end{split}
\end{equation*}
Dropping the $\epsilon$ term coming from the Laplacian, we obtain
\begin{align*}
  \ddt \|\st\|_{\dot{H}^j}^2 \leq C \big( \|\st\|_{\dot{H}^j}^2 + \|\st\|_{\dot{H}^j} \|\nabla \st\|_{\infty} + \|\st\|_{\dot{H}^j} \|\st\|_{\infty} \big), \hspace{5mm} \text{for all } 0 \leq j \leq k.
\end{align*}
Since $k > n/2 + 2$, 
\begin{align}
  \label{ineq:energy_estimate_strength}
  \ddt \|\st\|_{H^k}^2 \leq C \|\st\|_{H^k}^2. 
\end{align}
We will now denote the explicit dependencies on $\epsilon$ and take $\epsilon \to 0$. 
From Gr{\"o}nwall, we conclude that $Z^{\epsilon}(t)$ exists on a common time interval independent of $\epsilon$. 
Writing the equation for $(\ddt \steps)^2$, we have 
\begin{equation*}
  \big\| \ddt \steps \big\|_{L^2}^2 \leq C \|\steps\|_{H^1} + \epsilon \|\steps\|_{H^2} \leq C \|\steps\|_{H^k}. 
\end{equation*}
By local well-posedness of \eqref{SM_microscopic_viscous},  $Z^{\epsilon} \in C([0,T]; R^{2nN} \times H^k(\T^n))$ and therefore $\frac{d}{dt} Z^{\epsilon} \in L^2([0,T]; \R^{2nN} \times L^2(\T^n))$.  
By the Aubin-Lions lemma, we obtain a subsequence, which we denote again by $Z^{\epsilon}$ such that $Z^{\epsilon} \to Z^0$ in $C([0,T]; R^{2nN} \times H^{k-1}(\T^n))$. 
Since $H^{k-1}(\T^n)$ is dense in $H^k(\T^n)$, we have $Z^0 \in C_w([0,T]; \R^{2nN} \times H^k(\T^n))$.  Finally, since $k > n/2 + 2$, the terms $A^{\epsilon}$ converge 
pointwise to $A$. Taking $\epsilon \to 0$ in the Duhamel formula \eqref{eqn:Duhamel_formula}, we get 
\begin{align*}
  Z^0(t) = Z^0_0 + \int_0^t A(\tau) d\tau .
\end{align*}
That is, $Z^0 \in C_w([0,T]; \R^{2nN} \times H^k(\T^n))$ solves \eqref{SM_microscopic}. 
Finally, we note that due to the $\epsilon$-independent energy estimate \eqref{ineq:energy_estimate_strength}, $\|Z^0\|_X$ remains bounded for any finite time and thus exists for all time. 
This concludes existence and uniqueness of solutions to \eqref{SM_microscopic} on the global time interval $[0, \infty)$.


\section{Mean Field Limit}
\label{sec:mean_field_limit} The purpose of this section is to  establish the passage  from the discrete system \eqref{SM_microscopic} to the kinetic system \eqref{SM_kinetic} and, in particular, to prove Theorem \ref{thm:mean_field_limit}.  We will skip some details of the argument that appeared already for the Cucker-Smale model, as written in \cite{S-book} and focus mostly on the new ingredient pertaining to the adaptive strength.

First, let us define a weak version of \eqref{SM_kinetic}. 
\begin{definition}
  \label{defn_weak_soln}
  Fix a time $T > 0$ and an integer $k \geq 0$. We say the pair $(\mu, \st)$ with 
  $\mu \in C_{w^*}([0,T]; \cP(\T^n \times \R^n))$ and $\st \in C([0,T]; C^k(\T^n))$ is a \textit{weak} solution to \eqref{SM_kinetic} 
  if for all $g \in C_0^{\infty}([0,T] \times \T^n \times \R^n$) and for all $0 < t < T$, 
  \begin{equation}
      \label{SM_kinetic_weak}
      \begin{cases}
          \int_{\T^n \times \R^n} g(t,x,v) d\mu_t(x,v) &= \int_{\T^n \times \R^n} g(0,x,v) d\mu_0(x,v) \\
          & \quad + \int_0^t \int_{\T^n \times \R^n} (\partial_{\tau} g + v \cdot \nabla_x g \\
          & \quad + \hspace{1mm}\st (\uavg - v) \cdot \nabla_v g) d\mu_{\tau}(x,v) \\
          \partial_t \st + \nabla_x \cdot (\st \uavg) = 0. 
      \end{cases}
  \end{equation}
  In other words, $\mu$ solves the first equation weakly and the strength, $\st$, solves the second equation strongly. 
\end{definition}
As in the Cucker-Smale case, the empirical measure 
\begin{equation}\label{defn:sm_empirical_measures}
 \mu^N_t = \sum_{i=1}^N m_i \delta_{x_i(t)} \otimes \delta_{v_i(t)} 
\end{equation}
is a solution to \eqref{SM_kinetic_weak} if and only if $(x_i(t), v_i(t), \st(t,x))$ solve the discrete system \eqref{SM_microscopic}.
The well-posedness of \eqref{SM_kinetic_weak} for empirical measure-valued solutions is therefore equivalent to the well-posedness of the discrete-continuous system \eqref{SM_microscopic} established in Section \ref{sec:well-posedness_of_sm_microscopic}.
To show existence of general weak solutions to \eqref{SM_kinetic_weak}, we will show that the weak solution arises as a weak limit of \eqref{defn:sm_empirical_measures}.

First we note by a similar argument to the Cucker-Smale case, that $\mu_t$ is the push-forward of $\mu_0$ along the characteristic flow: 
\begin{subequations}\label{eqn:sm_characteristic_eqns}
\begin{align}
      \ddt X(t,s,x,v) &= V(t,s,x,v), & X(s,s,x,v) &= x \\
      \ddt V(t,s,x,v) & = \st(X) (\uavg(X) - V), & V(s,s,x,v) &= v.
\end{align} 
\end{subequations}

Thus,
\begin{align}
  \label{eqn:conservation_law}
  \int_{\T^n \times \R^n} h(X(t,\o), V(t,\o)) \dmu_0(\o) = \int_{\T^n \times \R^n} h(\o) \dmu_t(\o), \qquad \o = (x,v)
\end{align}

Owing to the maximum principle for the velocity characteristics, we have a non-expansion of the support of the measure-solutions in $v$: if $\supp{\mu_0} \subset \T^n \times B_R$, then $\supp{\mu_t} \subset \T^n \times B_R$ for all $t \in [0,T)$.   One immediate consequence of this is that $J \leq \|u\|_{L^\infty(\rho)} \leq R$ uniformly in time. Hence, the regularity estimates of Lemmas \ref{prop:regularity_velocity_averaging} and \ref{prop:regularity_strength} are available. 

The goal will be to establish a stability estimate:  for any two weak solutions $(\mu', \st')$ and $(\mu'', \st'')$ in the sense of Definition \ref{defn_weak_soln}, there exists a constant $C(R,T)$ such that 
\begin{align}
    \label{stability_Wasserstein}
    &W_1 (\mu_t', \mu_t'') \leq C_1(R,T) W_1(\mu_0', \mu_0'') .
\end{align}
For then, a Cauchy sequence $\mu_0^N$ with $W_1(\mu_0^N, \mu_0) \to 0$ yields convergence of $\mu_t^N$ to some $\mu \in C_{w^*}([0,T]; \cP(\T^n \times B_R))$.  
We will also establish stability with respect to the strength, \eqref{reg2_st'}, which will yield the convergence of $\st^N$ to some $\st \in C([0,T]; C^k(\T^n))$. 
Finally, we will verify that $(\mu, \st)$ is a weak solution to \eqref{SM_kinetic_weak} in Lemma \ref{lmma:limit_is_a_soln}.

\begin{lemma} (Deformation Tensor Estimates)
  \label{deformation_estimates} 
  Let $(\mu, \st)$ be a weak solution to \eqref{SM_kinetic_weak} on $[0,T]$ with characteristics $X,V$ given in \eqref{eqn:sm_characteristic_eqns}.  
  Then
  \begin{align*}
    \| \nabla X \|_{\infty} + \| \nabla V \|_{\infty} \leq C(R,T).
  \end{align*}
  \begin{proof}
    Differentiating \eqref{eqn:sm_characteristic_eqns},
   \begin{equation*}\label{}
\begin{split}
 \ddt \nabla X &= \nabla V \\ 
        \ddt \nabla V &= \nabla X^T \nabla \st(X) (\uavg(X) - V) + \st(X) \nabla X^T \nabla \uavg(X) - \st(X) \nabla V.
\end{split}
\end{equation*}    
    
    By the maximum principle on the velocity and the inherited regularity \eqref{reg1_uavg}, \eqref{reg2_uavg}, and \eqref{reg1_st}
    \begin{align*}
      \ddt (\|\nabla V\|_{\infty} + \|\nabla X\|_{\infty}) \leq C' ( \|\nabla X\|_{\infty} + \|\nabla V\|_{\infty} ) .
    \end{align*}
    We conclude by Gr{\"o}nwall. 
  \end{proof}
\end{lemma}

\begin{lemma} (Continuity Estimates)
  \label{characteristic_estimates}
  Let $(\mu', \st'), (\mu'', \st'')$ be weak solutions to \eqref{SM_kinetic_weak} on $[0,T]$;
  and let $X', X'', V', V''$ be the corresponding characteristics given by \eqref{eqn:sm_characteristic_eqns}. Then  
  \begin{align*}
    \|X' - X''\|_{\infty} + \|V' - V''\|_{\infty} \leq C( R, T) W_1(\mu_0', \mu_0'').
  \end{align*}
  \begin{proof}
    We have 
\begin{equation*}\label{}
\begin{split}
            \ddt (X' - V') &= V' - V'' \\
            \ddt (V' - V'') &= \st'(X') (\uavgp(X') - V') - \st''(X'') (\uavgpp(X'') - V'')  .
\end{split}
\end{equation*}
    By the maximum principle on the velocity and inherited regularity conditions \eqref{reg1_uavg}, \eqref{reg2_uavg}, \eqref{reg1_st}, \eqref{reg2_st}, 
    \begin{align*}
      \ddt \|V' - V''\|_{\infty} 
        &\leq \|\st'(X') - \st'(X'')\|_{\infty} \| \uavgp(X') - V' \|_{\infty} \\
        &\quad +\|\st'(X'') - \st''(X'')\|_{\infty} \| \uavgp(X') - V' \|_{\infty} \\
          &\quad + \|\st''(X'')\|_{\infty} \| \| \uavgpp(X') - V' - \uavgpp(X'') + V'' \|_{\infty} \\
        &\leq C \Big( \|V' - V''\|_{\infty} + \sup_{t \in [0,T]} (W_1(\rho', \rho'') + W_{\BL}(\u'\rho',  \u''\rho''))(t) \Big).
    \end{align*}
    Combining with $\ddt \|X' - X''\|_{\infty} \leq \|V' - V''\|_{\infty}$, we get
    \begin{equation}
        \label{charac_estimate_intermediate}
        \begin{split}
        \ddt \big( \|X' &- X''\|_{\infty} + \|V' - V''\|_{\infty} \big) \\
            &\leq C \Big( \|X' - X''\|_{\infty} + \|V' - V''\|_{\infty}  \\
            &\quad + \sup_{t \in [0,T]} (W_1(\rho', \rho'') + W_{\BL}(\u'\rho',  \u''\rho''))(t) \Big) . 
            \end{split}
    \end{equation}
    To estimate $W_1(\rho',\rho'')$ and $W_{\BL}(\u' \rho', \u'' \rho'')$, we use the fact that $\mu_t$ is the push-forward of $\mu_0$, \eqref{eqn:conservation_law}. Fix $\|g\|_{Lip} \leq 1$, then
    \begin{align*}
        \int_{\T^n} g(x) (d\rho_t' - d\rho_t'') 
            &= \int_{\T^n \times \R^n} g(x) (d\mu_t' - d\mu_t'') \\
            &= \int_{\T^n \times \R^n} g(X') d\mu_0' - \int_{\T^n \times \R^n} g(X'') d\mu_0''  \\
            &= \int_{\T^n \times \R^n} (g(X') - g(X'')) d\mu_0' + \int_{\T^n \times \R^n} g(X'') (d\mu_0' - d\mu_0'') \\
            &\leq  \|X' - X''\|_{\infty} + \|\nabla X''\|_{\infty} W_1(\mu_0', \mu_0'').
    \end{align*}
    For $W_{\BL}(\u' \rho', \u'' \rho'')$, fix $\|g\|_{\infty}, \|g\|_{Lip} \leq 1$, we have
    \begin{align*}
        & \int_{\T^n} g(x) (d(\u_t '\rho_t') - d(\u_t'' \rho_t'')) \\
            &= \int_{\T^n \times \R^n} vg(x) (d\mu_t' - d\mu_t'') \\
            &= \int_{\T^n \times \R^n} V'g(X') d\mu_0' - \int_{\T^n \times \R^n} V''g(X'') d\mu_0''  \\
            &= \int_{\T^n \times \R^n} (V'g(X') - V''g(X'')) d\mu_0' + \int_{\T^n \times \R^n} V''g(X'') (d\mu_0' - d\mu_0'') \\
            &\leq  \|g\|_{\infty} \|V' - V''\|_{\infty} + M \|X' - X''\|_{\infty} R \\
            &\quad + \big( \|g\|_{\infty} \|\nabla V''\|_{\infty} + R \|\nabla X''\|_{\infty} \big) W_1(\mu_0', \mu_0'').
    \end{align*}  
    These estimates hold uniformly in time. Plugging these into \eqref{charac_estimate_intermediate} and using Lemma \ref{deformation_estimates}, we conclude by Gr{\"o}nwall. 
  \end{proof}
\end{lemma}
An immediate Corollary is that the regularity conditions \eqref{reg2_uavg} and \eqref{reg2_st} can be restated in terms of the 
distance $W_1(\mu_0',\mu_0'')$. 
\begin{corollary}
  \label{corollary_continuity_estimates}
  Let $(\mu', \st'), (\mu'', \st'')$ be weak solutions to \eqref{SM_kinetic_weak} on $[0,T]$.
  Then
  \begin{align*}
    \sup_{t \in [0,T]} \big( W_1(\rho', \rho'') + W_{\BL}(\u'\rho', \u''\rho'') \big) \leq C_1(R,T) W_1(\mu_0', \mu_0''),
  \end{align*}
  and 
  \begin{align*}
    \label{reg2_uavg'}
    \tag{$\u$Lip2} &\sup_{t \in [0,T]} \|\uavgp - \uavgpp\|_{C^k} \leq C_2(k,R,T) W_1(\mu_0', \mu_0'') , \\
    \label{reg2_st'}
    \tag{$\st$Lip2} &\sup_{t \in [0,T]}  \|\st' - \st''\|_{C^k} \leq C_3(k,R,T) \Big( \|\st_0' - \st_0''\|_{C^k}  + W_1(\mu_0', \mu_0'') \Big) . 
  \end{align*}
\end{corollary}

Letting $X' := X'(t,\o)$, $V' := V'(t,\o)$ and similarly for $X''$, $V''$, we have for any $h \in Lip(\T^n \times \R^n)$ with $Lip(h) \leq 1$, 
\begin{align*}
  &\int_{\T^n \times \R^n} h(\o) (d\mu_t' - d\mu_t'') \\
      &= \int_{\T^n \times \R^n} h(X', V') d\mu_0'  - \int_{\T^n \times \R^n} h(X'', V'') d\mu_0'' \nonumber \\
      &= \int_{\T^n \times \R^n} h(X', V') (d\mu_0' - d\mu_0'') +\int_{\T^n \times \R^n} (h(X', V') - h(X'', V'')) d\mu_0'' \nonumber \\
      &\leq (\|\nabla X'\|_{\infty} + \|\nabla V'\|_{\infty}) W_1(\mu_0', \mu_0'')  \nonumber
        + \|X' - X''\|_{\infty} + \|V' - V''\|_{\infty} .
\end{align*}
Therefore, 
\begin{equation}
  \label{W1_estimate1}
  \begin{split}
  	W_1(\mu_t', \mu_t'') 
    	&\leq (\|\nabla X'\|_{\infty} + \|\nabla V'\|_{\infty}) W_1(\mu_0', \mu_0'') \\  
    	&\quad + \|X' - X''\|_{\infty} + \|V' - V''\|_{\infty} . 
    \end{split}
\end{equation}

Lemmas \ref{deformation_estimates}, \ref{characteristic_estimates} and inequality \eqref{W1_estimate1} imply the desired Wasserstein-1 stability \eqref{stability_Wasserstein}.
We conclude that the empirical measures $\mu_t^N$ converge in the Wasserstein-1 metric to some $\mu \in C_{w^*}([0,T]; \cP(\T^n \times B_R))$ uniformly on $[0,T]$
(the weak$^*$ continuity of $\mu$ owes to the weak$^*$ continuity of the empirical measures and the uniform convergence on $[0,T]$).
In addition, by Corollary \ref{corollary_continuity_estimates}, for any $k\geq 0$, $\st_t^N$ converges in $C^k$ to some $\st_t \in C^{k}(\T^n)$ uniformly on $[0,T]$. 
It remains to show that the limiting pair $(\mu, \st)$ is in fact a weak solution to \eqref{SM_kinetic} in the sense of Definition \ref{defn_weak_soln}. 

\begin{lemma}
    \label{lmma:limit_is_a_soln}
    Suppose a sequence $\mu^N \in C_{w^*}([0,T]; \cP(\T^n \times B_R))$ converges weakly pointwise, i.e. $\mu_t^N \to \mu_t$  for all $t \in [0,T]$; and suppose that 
    for any $k\geq 0$, $\st_t^N$ converges in $C^k(\T^n)$ to $\st_t$, i.e. $\|\st_t^N - \st_t \|_{C^k} \to 0$, uniformly for all $t \in [0,T]$. 
    Then $(\mu, \st) \in C_{w^*}([0,T]; \cP(\T^n \times B_R)) \times C([0,T], C^{k}(\T^n))$ is a weak solution to \eqref{SM_kinetic}. 
\end{lemma}
\begin{proof}
    We have already observed that, due to Corollary \eqref{corollary_continuity_estimates}, for any $k$, $\st_t^N$ converges in $C^k$ to some $\st_t \in C^{k}(\T^n)$ 
    uniformly on $[0,T]$. In addition $\partial_t \st^N = -\nabla \cdot (\st^N \uavgN)$ is uniformly bounded on $[0,T]$ due to \eqref{reg1_uavg} and \eqref{reg1_st}. 
    So, $\st \in C([0,T]; C^k(\T^n))$. 
    To show that $\st$ solves the transport equation $\partial_t \st + \nabla \cdot (\st \uavg) = 0$, let $\tilde{X}$ be the characteristic flow, 
    \begin{align*}
      \ddt \tilde{X}(t, \alpha) = \uavg(t, \tilde{X}(t,\alpha)), \hspace{5mm} \tilde{X}(0, \alpha) = \alpha .
    \end{align*}
    Similarly, let $\tilde{X}^N$ be the characteristic flow for the empirical strength $\st_t^N$ which solves $\partial_t \st^N + \nabla \cdot (\st^N \uavgN) = 0$,
    \begin{align*}
      \ddt \tilde{X}^N(t, \alpha) = \uavgN(t, \tilde{X}^N(t,\alpha)), \hspace{5mm} \tilde{X}^N(0, \alpha) = \alpha .
    \end{align*}
    Abbreviating $\tilde{X}(t,\alpha), \tilde{X}^N(t,\alpha)$ by $\tilde{X}, \tilde{X}^N$ and using \eqref{reg1_uavg} and \eqref{reg2_uavg'}, we have for all $k \geq 0$
    \begin{align*}
      \|\partial^k &(\uavgN (\tau, \tilde{X}^N) - \uavg (\tau, \tilde{X})) \|_{\infty} \\
        &\leq \|\partial^k (\uavgN (\tau, \tilde{X}) - \uavg (\tau, \tilde{X}))\|_{\infty}  \nonumber
        + \|\partial^k (\uavgN (\tau, \tilde{X}^N) - \uavgN (\tau, \tilde{X})) \|_{\infty} \nonumber \\
        &\leq \|\partial^k (\uavgN (\tau, \tilde{X}) - \uavg (\tau, \tilde{X})) \|_{\infty}  \nonumber
        + \|\partial^{k+1} \uavgN\|_{\infty} \| \tilde{X}^N - \tilde{X} \|_{\infty} \nonumber \\
        &\leq C \big( W_1(\mu_0^N, \mu_0) + \| \tilde{X}^N - \tilde{X} \|_{\infty} \big). \nonumber
    \end{align*}
    As a result, 
    \begin{align*}
      \ddt \|(\tilde{X}^N - \tilde{X})(t,\cdot)\|_{\infty} \leq C \big( W_1(\mu_0', \mu_0'') + \| (\tilde{X}^N - \tilde{X}) (t,\cdot) \|_{\infty} \big), \hspace{5mm} \text{for all } t \in [0,T) .
    \end{align*}
    By Gr{\"o}nwall and $\tilde{X}(0,\alpha) = \tilde{X}^N(0,\alpha) = \alpha$, we obtain 
    \begin{align*}
      \|(\tilde{X}^N- \tilde{X})(t,\cdot)\|_{\infty} \leq C' W_1(\mu_0^N, \mu_0), \quad \text{for all } t \in [0,T) .
    \end{align*}
    Now solving along the characteristic $\tilde{X}^N$, 
    \begin{align*}
      \st^N (t, \tilde{X}^N(t, \alpha)) = \st_0^N(\alpha) \exp \Big\{ \int_0^t \nabla \cdot \uavgN (\tau, \tilde{X}^N(s, \alpha)) d\tau \Big\}.
    \end{align*}
    With these estimates in hand, we obtain uniform in time convergence to 
    \begin{align*}
      \st (t, \tilde{X}(t, \alpha)) = \st_0(\alpha) \exp \Big\{ \int_0^t \nabla \cdot \uavg (\tau, \tilde{X}(s, \alpha)) d\tau \Big\}.
    \end{align*}
    In particular, $\st$ is a solution to $\partial_t \st + \nabla(\st \uavg) = 0$. 
    Turning to the convergence of the empirical measures, the weak convergence $W_1(\mu_t^N, \mu_t) \to 0$ immediately implies that the linear terms in \eqref{SM_kinetic_weak} converge.  
    Let us address the nonlinear term.
    \begin{align*}
        \int_0^t &\int_{\T^n \times \R^{n}} \nabla_v g \st^N (\uavgN - v)) d\mu^N_{\tau}(x,v) - \int_0^t \int_{\T^n \times \R^{n}} \nabla_v g \st (\uavg - v)) d\mu_{\tau}(x,v) \\
            &\leq \|\nabla_v g\|_{\infty} \int_0^T \int_{\T^n \times \R^{n}}  \|\st^N (\uavgN - v)) - \st (\uavg - v))\|_{\infty} d\mu^N_{\tau}(x,v) \\
            &\quad + \|\nabla_v g\|_{\infty} \int_0^T \int_{\T^n \times \R^{n}}  \| \st (\uavg - v))\|_{\infty} d(\mu^N_{\tau}(x,v) - \mu_{\tau}(x,v)).
    \end{align*}
    The second term goes to zero by weak convergence.  For the first term, we simply use the regularity conditions \eqref{reg1_uavg}, \eqref{reg2_uavg'}, \eqref{reg1_st}, \eqref{reg2_st'} to get
    \begin{align*}
        &\|\st^N (\uavgN - v)) - \st (\uavg - v))\|_{\infty} \\
            &\leq \| \st^N - \st \|_{\infty} \|\uavgN\|_{\infty} + \| \st \|_{\infty} \|\uavgN - \uavg\|_{\infty} \\
            &\leq C W_1(\mu_0^N, \mu_0).
    \end{align*}
    This yields $\|\st^N (\uavgN - v)) - \st (\uavg - v))\|_{\infty} \to 0$ uniformly on $[0,T]$. 
\end{proof}


\section{Hydrodynamic Limits}
\label{sec:hydrodynamic_limits}
The goal of this section is to prove Theorems \ref{thm:monokinetic_limit} and \ref{thm:maxwellian_limit}, which establish a passage from the kinetic description \eqref{SM_kinetic} to the 
corresponding macroscopic description in the monokinetic and Maxwellian limiting regimes. 
The arguments presented here resemble the conventional case as done for general environmental averaging models in \cite{S-EA}. However, the analysis of the adaptive strength requires us to make substantial changes and additions throughout. We therefore present full details for completeness. 

From the introduction, the $v$-moments of the mesoscopic system \eqref{SM_kinetic} yields \eqref{SM_vmoments} with the Reynolds Stress tensor $\cR$ given by \eqref{defn:reynolds_stress_monokinetic}. 
The system is closed by adding a strong alignment force $\frac{1}{\e}F(\feps)$ to the kinetic equation \eqref{SM_kinetic}.
We will consider two variants of such an alignment force corresponding to the monokinetic and Maxwellian regimes. Let us start with the former.

\subsection{Monokinetic Limit}
\label{sec:monokinetic_regime} 

We consider $\eqref{SM_kinetic}$ under the strong local alignment force (setting $\l=1$ for convenience)
\begin{align}
  \label{SM_kinetic_monokinetic}
    \begin{cases}
        \partial_t \feps + v \cdot \partial_x \feps = \nabla_v \cdot (\steps (v - \uavgeps) \feps) + \frac{1}{\epsilon} \nabla_v \cdot ((v - \ueps_{\delta} )\feps) \\
        \partial_t \steps + \nabla \cdot (\steps \uavgeps) = 0  ,
    \end{cases}
\end{align}
where $\u_{\delta}$ is the special mollification given by
\begin{align}
    \label{special_mollification}
    \u_{\delta} = \Big( \frac{(\u \rho)\ast {\psi_{\delta}}}{\rho\ast {\psi_{\delta}}} \Big)\ast {\psi_{\delta}}
\end{align}
for some smooth mollifier $\psi_{\delta}(x) = \frac{1}{\delta^n} \psi(x/\delta)$, $\psi >0$. Note that this mollification resembles the averaging protocol of the \ref{Mf}-model. It was introduced in \cite{S-hypo} to circumvent the problem of roughness of the macroscopic field $\u$ and the non-uniqueness of the corresponding characteristic flow as in \cite{FK2019}. As a result, the hydrodynamic limit was extended to vacuous solutions, as it will be done here as well.  

This special mollification has several remarkable properties. First, since $\psi>0$, $\u_\d \in C^\infty$ for any $\d>0$. Therefore, the system \eqref{SM_kinetic_monokinetic} is globally well-posed. Second, the mollification has a key approximation property stating that $\u_{\delta}$ is close to $\u$ for small $\delta$ with a bound independent of $\rho$. 
The following approximation lemma can be found in \cite[Lemma 9.1]{S-EA}.
\begin{lemma}
    \label{lmma:approximation_property}
    For any $\u \in Lip$ and for any $1 \leq p < \infty$, one has 
    \begin{align*}
        \|\u_{\delta} - \u\|_{L^p(\rho)} \leq C \delta \|\u\|_{Lip}
    \end{align*}
    where $C > 0$ depends only on $\psi$ and $p$. 
\end{lemma}

We will see that the optimal resolution scale $\d$ in the limit as $\e\to 0$, is in fact $\d \sim \e^2$. However, we will state the result for $\d,\e$ assumed independent so as to see where the optimal resolution is coming from.
 
The local alignment force pushes the solution towards the monokinetic distribution
\begin{equation}\label{monokinetic_ansatz}
  f(t,x,v) = \rho(t,x) \delta(v - \u(t,x)),
\end{equation}
where $(\rho, \st, \u)$ solve the pressureless $\st$-Euler-alignment system \eqref{SM_macroscopic_pressureless}.

\begin{theorem} (Monokinetic Limit)
  \label{thm:monokinetic_limit}
  Suppose the  kernel $\Phi$ satisfies regularity conditions \eqref{reg1_kernel} and \eqref{reg2_kernel}. 
  Let $(\rho, \st, \u)$ be a smooth solution to \eqref{SM_macroscopic_pressureless} on $\T^n \times [0,T]$ and let $f$ be the corresponding monokinetic ansatz \eqref{monokinetic_ansatz}. 
  Let $(\feps, \steps)$ be a solution to \eqref{SM_kinetic_monokinetic}  with initial conditions $\steps_0 \in C^{\infty}(\T^n)$, $f_0^{\epsilon} \in C_0^k(\T^n \times \R^n)$ satisfying:
  \begin{itemize}
      \item[(i)] $\supp f_0^{\epsilon} \subset \T^n \times B_R$, for a fixed $R>0$,
      \item[(ii)]  $W_2(f_0^{\epsilon}, f_0) < \epsilon$,
      \item[(iii)] $\steps_0 = \st_0 \in C^{\infty}(\T^n)$. 
  \end{itemize}
  Then there exists a constant $C(R,T)$ such that: 
\begin{equation*}
     \sup_{t \in [0,T]} \|\steps_t - \st_t \|_{C^k} + W_2(f_t^{\epsilon}, f_t)  \leq C \sqrt{\epsilon + \frac{\delta}{\epsilon}}.
\end{equation*}
\end{theorem}


We will first verify the maximum principle on the velocity so that we can use the inherited regularity of the strength. 
The characteristic equations of \eqref{SM_kinetic_monokinetic} are given by: 
\begin{subequations}\label{eqn:characteristic_sm_kinetic_monokinetic}
\begin{align} 
      \ddt X(t,s,x,v) &= V(t,s,x,v), &  X(s,s,x,v)& = x \\
      \ddt V(t,s,x,v) &= \st(X)(\uavgeps(X) - V) + \frac{1}{\epsilon} (\ueps_{\delta} - V), \quad &  V(s,s,x,v) &= v. 
\end{align} 
\end{subequations}
As in the Cucker-Smale and $\st$-model, the measure $f_t(x,v) \dx \dv$ is the push-forward of $f_0(x,v) \dx \dv$ along the characteristic flow. 
Letting $\o := (x,v)$, $X' := X(t, \o')$, $V' := V(t, \o')$, and using the right stochasticity of $\Phi_{\rho}$, \eqref{eqn:kernel_right_stochastic}, we compute: 
\begin{align*}
  \st(X)(\uavgeps(X) - V) 
    &= \int_{\T^n \times \R^n} \st(X)  \Phi_{\rhoeps}(X, x) v \feps_t(x,v) \dx \dv - V \\
    &= \int_{\T^n \times \R^n} \st(X)  \Phi_{\rhoeps}(X, x) (v - V) \feps_t(x,v) \dx \dv \\
    &= \int_{\T^n \times \R^n} \st(X)  \Phi_{\rhoeps}(X, X') (V' - V) \feps_0(\o') \domega' . 
\end{align*}
Considering compactly supported initial data, $\supp f_0 \subset \T^n \times B_R$, and evaluating at a point of maximum, $V_+(t) = \max_{(x,v) \in \T^n \times B_R} |V(t,0,x,v)|$, we have 
\begin{align*}
  &\int_{\T^n \times \R^n} \st(X) \Phi_{\rhoeps}(X, X') (V' - V_+) \feps_0(x,v) \domega' \leq 0.
\end{align*}
For the local alignment term, recall that $\ueps_{\delta}$ is just another averaging given by a kernel  $\Phi_{\rhoeps, \delta}$ as in Table~\ref{t:kernels}, we have: 
\begin{align*}
  \frac{1}{\epsilon} \int_{\T^n \times \R^n} \Phi_{\rhoeps, \delta}(X, X') (V' - V_+) f_0(x,v) \domega' \leq 0 .
\end{align*}
Using the classical Rademacher Lemma, we obtain 
\begin{align*}
  \ddt \|V\|_{\infty} \leq 0 .
\end{align*}
This implies that if initially $\supp f_0^{\epsilon} \subset \T^n \times B_R$, then $\supp f_t^{\epsilon} \subset \T^n \times B_R$ for all $t>0$. As a consequence, we obtain uniform boundedness of the macroscopic velocities:
\begin{align*}
 | \ueps | = \frac{ \left| \int_{B_R} v \feps dv \right| }{\int_{B_R} \feps dv} \leq \frac{ R \int_{B_R} \feps dv}{\int_{B_R} \feps dv} = R ,
\end{align*}
and as a therefore, $J^\e,J \leq R T$.

Now that the $J$-terms are under control, the inheritance lemmas stated in \sect{sec:inherited_reg_from_the_kernel}  become available. As a consequence of \eqref{reg1_uavg}, \eqref{reg1_st}, we have  uniform bounds on the averaged velocities and strengths
\begin{equation}\label{e:usunif}
  \|\uavgeps \|_{C^k} + \|\steps \|_{C^k} \leq C.
\end{equation}

Next, we rephrase \eqref{reg2_uavg}, \eqref{reg2_st} in terms of modulated macroscopic energy and $W_1$-metric.
To that end, let us estimate the distance between the momenta. Let us fix an arbitrary $\|g\|_\infty, \|\n g\|_\infty \leq 1$, then
\begin{equation*}\label{}
\begin{split}
\int_{\T^n} g (\ueps \rhoeps - \u\rho)\dx &= \int_{\T^n} g (\ueps - \u) \rhoeps \dx + \int_{\T^n} g  \u(\rhoeps -\rho)\dx \\
& \leq \left( \int_{\T^n} |\ueps - \u|^2 \rhoeps \dx \right)^{1/2} + \|\n \u\|_\infty W_\BL(\rhoeps,\rho).
\end{split}
\end{equation*}
Since $\|\n \u\|_\infty$ is uniformly bounded on the given time interval, we obtain  
\begin{equation*}
W_\BL(\ueps \rhoeps, \u\rho) \lesssim \left( \int_{\T^n} |\ueps - \u|^2 \rhoeps \dx \right)^{1/2}+W_\BL(\rhoeps,\rho).
\end{equation*}
We also trivially have $W_\BL(\rhoeps,\rho) \leq W_1(\rhoeps,\rho)$. 

Consequently, by the inheritance lemmas:
\begin{subequations}\label{e:reg2'}
\begin{align}
    \label{ureg2'} 
     \|\uavgeps - \uavg\|_{C^k} & \lesssim  \left( \int_{\T^n} |\ueps - \u|^2 \rhoeps \dx \right)^{1/2}+W_1(\rhoeps,\rho), \\
    \label{sreg2'} 
  \|\steps - \st\|_{C^k} & \lesssim \int_0^t \left( \left( \int_{\T^n} |\ueps - \u|^2 \rhoeps \dx \right)^{1/2}+W_1(\rhoeps,\rho) \right)\dtau.
\end{align}
\end{subequations}

\begin{proof} (proof of Theorem \ref{thm:monokinetic_limit})
  In order to control $W_2(\feps, f)$, we consider the flow $t\to \gamma_t$ whose marginals are $\feps_t$ and $f_t$ (i.e. $\gamma_t \in \Pi(\feps_t, f_t)$)
    \begin{equation*}
    \begin{split}
        &\partial_t \gamma + v_1 \cdot \nabla_{x_1} \gamma + v_2 \cdot \nabla_{x_2} \gamma \\
            &+ \nabla_{v_1} [ \gamma \steps (v_1 - \uavgeps) + \frac{1}{\epsilon} (v_1 - \ueps_{\delta}) ]
            + \nabla_{v_2} [ \gamma \st (v_2 - \uavg) ] = 0.
    \end{split}
    \end{equation*}
    Since $\gamma_t \in \Pi(\feps_t, f_t)$,  
    \begin{align*}
        W := \int_{\T^{2n} \times \R^{2n}} |w_1 - w_2|^2  \mbox{d} \gamma_t(w_1, w_2) \geq W_2^2(\feps_t, f_t).
    \end{align*}
    So, we aim to control $W$.  Splitting it into the potential and kinetic components, we get
    \begin{align*}
        W = \int_{\T^{2n} \times \R^{2n}} |v_1 - v_2|^2 \mbox{d} \gamma_t + \int_{\T^{2n} \times \R^{2n}} |x_1 - x_2|^2 \mbox{d} \gamma_t := W_v + W_x .
    \end{align*} 
Since the $(x_1,x_2)$-marginal of $\g$ belongs to $\Pi(\rho^\e,\rho)$ we  find 
\begin{equation}\label{e:W1Wx}
W^2_1(\rho^\e,\rho)\leq W^2_2(\rho^\e,\rho) \leq W_x.
\end{equation}

Central to the proof is the modulated energy given by
 \begin{align}
  \label{defn:modulated_kinetic_energy}
  \rme(\feps | \u) = \int_{\T^n \times \R^n} | v - \u(x)|^2 \feps(x,v) \dv \dx .
\end{align}
We now recall three inequalities which do not involve the adaptive strength so they carry over from the classical case. First, we have, see \cite[(9.9)]{S-EA},
\begin{equation}\label{e:modemace}
\rme(f^\e| \u) \geq  \int_{\O}  \rho^\e| \u^\e - \u|^2 \dx.
\end{equation}
Second, we have the following system, see \cite[(9.10)]{S-EA},
    \begin{equation}
    \label{w2_kin_pot_estimate}
    \begin{split}
            \ddt W_x & \lesssim \emodu +  W_x \\
            W_v & \lesssim \emodu +  W_x.
    \end{split}
    \end{equation}
Third,  the evolution of the modulated kinetic energy can be estimated as, see \cite[page 391]{S-EA},
    \begin{align*}
        \ddt \emodu &\lesssim \emodu + \frac{\delta}{\epsilon} + (\ueps - \u, \uavgeps - \ueps)_{\kappa_{\rhoeps}} \\
        &\quad + \int_{\T^n} \rhoeps(\u - \ueps) \cdot \st (\uavg - \u) \dx .
    \end{align*}
   
 The last two terms can be written as  
    \begin{align*}
    A:=  &  (\ueps - \u, \uavgeps - \ueps)_{\kappa_{\rhoeps}} + \int_{\T^n} \rhoeps(\u - \ueps) \cdot \st (\uavg - \u) \dx  \\
          &= \int_{\T^n} \rhoeps (\u - \ueps) \cdot \big( \steps \uavgeps - \st \uavg + \st \u - \steps \ueps \big) \dx \\
        & =  \int_{\T^n} \rhoeps (\u - \ueps) \cdot ( \steps - \st) \uavgeps  \dx +\int_{\T^n} \rhoeps (\u - \ueps) \cdot \st (\uavgeps -\uavg ) \dx\\
        & \quad +\int_{\T^n}  \rhoeps (\u - \ueps) \cdot ( \st -\steps) \u \dx + \int_{\T^n}  \rhoeps |\u - \ueps|^2  \steps  \dx \\        
       & := I + II + III + IV .
    \end{align*}
 To continue   with the estimates, let us rewrite the inheritance estimates \eqref{ureg2'}- \eqref{sreg2'} using \eqref{e:W1Wx}, \eqref{e:modemace}:
 \begin{align*}
     \|\uavgeps - \uavg\|^2_{C^k} & \lesssim   \emodu +W_x, \\
  \|\steps - \st\|^2_{C^k} & \lesssim \int_0^t  ( \emodu +W_x) \dtau.
\end{align*}
 Using these along with \eqref{e:usunif}   repeatedly we obtain,
 \begin{equation*}\label{}
\begin{split}
 I & \lesssim \emodu +  \int_0^t  ( \emodu +W_x) \dtau \\
 II & \lesssim \emodu +  W_x \\
 III &\lesssim \emodu +  \int_0^t  ( \emodu +W_x) \dtau \\
 IV &\lesssim \emodu .
\end{split}
\end{equation*}
Summarizing, we obtain 
\begin{equation*}\label{ }
 \ddt \emodu \lesssim \frac{\delta}{\epsilon} + \emodu +  W_x +  \int_0^t  ( \emodu +W_x) \dtau.
\end{equation*}
All together, 
    \begin{align*}
            \ddt W_x & \lesssim \emodu + W_x \\
            \ddt \emodu &\lesssim  \frac{\delta}{\epsilon} +  \emodu +  W_x +  \int_0^t  ( \emodu +W_x) \dtau.
    \end{align*}
Adding up the two equations and denoting 
\[
X = \int_0^t  ( \emodu +W_x) \dtau, \quad Y = \emodu +W_x,
\]
we obtain
\begin{equation*}\label{}
\begin{split}
\dot{X} & = Y \\
\dot{Y} & \lesssim  \frac{\delta}{\epsilon} + X + Y.
\end{split}
\end{equation*}

Since initially $X_0 = 0$, $Y_0 \leq \e$, the  Gr{\"o}nwall Lemma implies 
    \begin{align*}
      \emodu + W_x \lesssim \epsilon + \frac{\delta}{\epsilon},
    \end{align*}
 and by \eqref{w2_kin_pot_estimate}, $W_v \lesssim \epsilon + \frac{\delta}{\epsilon}$.
\end{proof}


\subsection{Maxwellian Limit}
\label{sec:maxwellian_limit}

In this section we analyze the hydrodynamic limit for a system with strong Fokker-Planck-Alignment force:
  \begin{align}
    \label{SM_kinetic_maxwell_la}
    \begin{cases}
        \partial_t \feps + v \cdot \nabla_x \feps + \lambda \nabla_v \cdot (\steps (v - \uavgeps) \feps) + \frac{1}{\epsilon} \Big( \Delta_v \feps + \nabla_v \cdot [ \feps (\ueps_{\delta} - v) ] \Big) = 0 \\
        \partial_t \steps + \nabla_x \cdot (\steps \uavgeps) = 0.
    \end{cases}
  \end{align} 
Heuristically,  such a force pushes the solution towards a local Maxwellian
\begin{equation}
\label{formula:maxwellian}
f^\e \to   \mu(t,x) = \frac{\rho(t,x)}{(2\pi)^{n/2}} e^{-\frac{|v - \u(t,x)|^2}{2}}, \quad \text{ and } \st^\e \to \st,
\end{equation}
where the triple $(\rho, \st, \u)$ solves  the system
\begin{align}
  \label{SM_hydrodynamic_maxwellian}
  \begin{cases}
      \partial_t \rho + \nabla \cdot (\u \rho) = 0 \\
      \partial_t \st + \nabla \cdot (\st \uavg) = 0  \\
      \partial_t (\rho \u) + \nabla \cdot (\rho \u \otimes \u) + \nabla \rho = \rho \st (\uavg - \u) . 
  \end{cases}
\end{align}

Unfortunately, the regularity assumptions on the kernel $\Phi$ alone are not sufficient to justify the limit. This is due to the lack of a priori uniform bounds on the $J^\e$-norms. Such bounds would normally come from the energy / entropy inequality for the system \eqref{SM_kinetic_maxwell_la}, which in the conventional settings holds due to a priori known uniform boundedness of the prescribed strength-functions $\st_\rho$.  In our case, to get such control over $\st$ we need the energy to be bounded, which creates a circular argument.  So, let us suppose that we have a uniform bound on $\st^\e$:
\begin{equation}\label{e:seuniform}
\sup_{\e>0,t\leq T} \| \st^\e \|_\infty \leq C.
\end{equation}

Under this assumption, let us consider the equation for the entropy given by
\begin{equation}\label{eqn:kin_Heps}
\cH_{\epsilon} = \int_{\T^n \times \R^n} \big( \feps \log \feps + \frac{1}{2} |v|^2 \feps \big) \dv \dx.
\end{equation}
We have
 \begin{align*}
    \frac{d}{dt} \cH_{\epsilon} 
    &= - \frac{1}{\epsilon} \int_{\T^n \times \R^n}  \frac{|\nabla_v \feps +\feps (v - \ueps_{\delta})|^2}{\feps}\dv \dx - \frac{1}{\epsilon} [ (\ueps_{\delta}, \ueps)_{\rhoeps} - (\ueps_{\delta}, \ueps_{\delta})_{\rhoeps} ] \\
    & \quad - \int_{\T^n \times \R^n} \steps [ \nabla_v \feps \cdot (v - \uavgeps) + v \cdot (v - \uavgeps) \feps ] \dv \dx \\
    &:= A_1 + A_2 + A_3.
  \end{align*}
  $A_1$ is the modulated Fisher information, which is  negative, so it can be dropped. According to \cite{S-EA} (due to  the so-called ball-positivity of the \ref{Mf}-model), 
  $A_2$ is also negative. Lastly,
  \begin{align*}
    A_3 =  n \int_{\T^n \times \R^n} \steps \feps \dv \dx  -\int_{\T^n \times \R^n} \steps |v|^2 \feps \dv \dx + (\ueps, \uavgeps)_{\kappa^{\epsilon}} .
  \end{align*}
  Due to assumed boundedness of the strength-functions, \eqref{e:seuniform}, 
  \begin{align*}
    A_3 \leq c_1  + (\ueps, \uavgeps)_{\kappa^{\epsilon}}  \leq c_1 + c_2 \|\ueps\|_{L^2(\rhoeps)}^2 \leq c_1+ c_2 \cE_\e,
  \end{align*}
where  $\cE_{\epsilon} = \int |v|^2 \feps(x,v) \dv \dx$. There is a classical inequality, see for example \cite{S-EA} in our context, stating that
\[
\cE_{\epsilon} \leq 2 \cH_\e + C.
\]
Thus,  we arrive at
\[
  \frac{d}{dt} \cH_{\epsilon} \leq c_1 + c_2 \cH_\e,
  \]
  with $c_1,c_2>0$ independent of $\e$.  Consequently, 
\begin{equation}\label{e:EHunif}
\sup_{\e>0,t\leq T} \cH_\e + \cE_\e \leq C.
\end{equation}
This implies in particular, since $\cE_\e\geq  \|\ueps\|_{L^2(\rhoeps)}^2 $, that 
\begin{equation*}\label{e:Junif}
\sup_{\e>0} J_\e \leq C.
\end{equation*}
As a result, we have the inheritance estimates \eqref{e:reg2'} at our disposal.

In general, the condition \eqref{e:seuniform} is not guaranteed to hold automatically for solutions of \eqref{SM_kinetic_maxwell_la}. However, it does hold in the obvious manner for the $\wt$-model. Indeed, if the model is based on the Favre filtration with a non-degenerate communication kernel \eqref{e:nondeg}, recall that $\st = (\rho \ast \phi) \wt$, where $\wt$ solves a pure transport equation. 
If the initial data for strength is the same $\steps_0 = \st_0 \in C^{\infty}(\T^n)$, then $\wt_0^\e = \wt_0 \in L^\infty$. By the transport of $\wt$, $\wt^\e$ will remain uniformly bounded; in addition, $\rho^\e \ast \phi \leq \|\phi\|_\infty$. We therefore have \eqref{e:seuniform}  satisfied in this particular case.  In the hydrodynamic limit for $\st$-models, condition \eqref{e:seuniform}  is the only requirement that separates us from the general result. 

We state the convergence result on the level of relative entropy:
\begin{equation*}
  \label{defn:relative_entropy}
  \cH(\feps | \mu) = \int_{\T^n \times \R^n} \feps \log \frac{\feps}{\mu} \dv \dx .
\end{equation*}
Due to the classical \CK, for some constant $c > 0$, this would also imply convergence in $L^1$-norm:
\begin{align*}
  c \|\feps - \mu\|_{L^1(\T^n \times \R^n)} \leq \cH(\feps | \mu).
\end{align*}
Furthermore, defining the Maxwellian associated to solutions of \eqref{eqn:SM_maxwellian_epsilon},
\begin{equation*}\label{ }
\mueps(t,x) = \frac{\rhoeps(t,x)}{(2\pi)^{n/2}} e^{-\frac{|v - \ueps(t,x)|^2}{2}},
\end{equation*}
the relative entropy can then be split into the sum
\begin{align*}
  &\cH(\feps | \mu) = \cH(\feps | \mueps) + \cH(\mueps | \mu) \\
  &\cH(\mueps | \mu) = \frac{1}{2} \int_{\T^n} \rhoeps |\ueps - \u|^2 \dx + \int_{\T^n} \rhoeps \log(\rhoeps / \rho) \dx .
\end{align*}
So, again, due to the \CK, $\cH(\feps | \mu) \to 0$ implies 
\begin{align*}
  &\rhoeps \to \rho \\
  &\ueps \rhoeps \to \u \rho 
\end{align*}
in $L^1(\T^n)$, which in turn, due to the inheritance bounds \eqref{reg2_st} would imply $\|\steps - \st\|_{C^k} \to 0$.
\begin{theorem}
  \label{thm:maxwellian_limit} (Maxwellian Limit) Consider the $\wt$-model with non-degenerate communication kernel \eqref{e:nondeg}. 
  Let $(\rho, \st, \u)$ be a smooth, non-vacuous solution to \eqref{SM_hydrodynamic_maxwellian} on $\T^n \times [0,T]$ and let $\mu$ be the Maxwellian given in \eqref{formula:maxwellian}.
  Let $(\feps, \steps)$ be a solution to the kinetic system \eqref{SM_kinetic_maxwell_la} with initial data  $\steps_0 \in C^{\infty}(\T^n)$ and $f_0^{\epsilon} \in C_0^k(\T^n \times \R^n)$  satisfying:
  \begin{itemize}
      \item[(i)]  $\cH(\feps_0 | \mu_0) \to 0$ as $\epsilon \to 0$,  
      \item[(ii)] $\steps_0 = \st_0 \in C^{\infty}(\T^n)$.
  \end{itemize}
  Then for $\delta = o(\epsilon)$, 
\begin{equation*}\label{}
\begin{split}
    \sup_{t \in [0,T]} \cH(\feps | \mu) & \to 0 \\
    \sup_{t \in [0,T]} \|\steps - \st\|_{C^k} & \to 0, \quad \forall k\in \N.
\end{split}
\end{equation*}
\end{theorem}

\begin{proof} 

Breaking the relative entropy into the kinetic and macroscopic parts, we have
\begin{equation*}
  \label{eqn:rel_entropy_kinetic_macro}
  \cH(\feps | \mu) = \cH_{\epsilon} + \cG_{\epsilon} + \frac{n}{2}\log(2\pi), 
 \end{equation*} 
 where $\cH_{\epsilon}$ is the entropy given in \eqref{eqn:kin_Heps} and 
 \begin{equation*}
  \label{eqn:macro_Geps}
  \cG_{\epsilon} = \int_{\T^n} \big( \frac{1}{2} \rhoeps |\u|^2 - \rhoeps \ueps \cdot \u - \rhoeps \log \rho \big) \dx  .  
\end{equation*}
We then seek to estimate $\cH_{\epsilon}$ and $\cG_{\epsilon}$.  This is done in the same manner as in \cite{S-EA}, but we present the modifications in order to accommodate for the adaptive strength. 
The macroscopic system for the $\e$-quantities is given by:
\begin{equation}
  \label{eqn:SM_maxwellian_epsilon}
  \begin{cases}
      \partial_t \rhoeps + \nabla \cdot (\ueps \rhoeps) = 0 \\
      \partial_t \steps + \nabla \cdot (\steps \uavgeps) = 0 \\ 
      \partial_t (\rho \ueps) + \nabla \cdot (\rho \ueps \otimes \ueps) + \nabla \rhoeps
        + \nabla \cdot \Reps \\ 
        \quad = \rhoeps \steps (\uavgeps - \ueps) 
          +\frac{1}{\epsilon} \rhoeps (\ueps_{\delta} - \ueps), 
  \end{cases}
\end{equation}
where $\ueps_{\delta}$ is given in \eqref{special_mollification} and  
\begin{align*}
  \Reps(t,x) = \int_{\T^n} ((v-\ueps) \otimes (v - \ueps) - \mathbb{I}) \feps \dv.
\end{align*}
We have, see \cite[(9.21)]{S-EA},
\begin{equation*} \label{ineq:Heps}
  \frac{d}{dt} \Heps 
    \leq 
  -\frac{1}{\epsilon} \Ieps + \frac{\epsilon}{4} \int_{\T^n \times \R^n} \steps |v - \ueps|^2 \feps \dv \dx 
      - \|\ueps\|_{L^2(\kappa_{\epsilon})^2} + (\ueps, \uavgeps)_{\kappa^{\epsilon}} ,
\end{equation*}
where the relevant the Fisher information, $\Ieps$, is given by
\begin{align}
  \label{defn:Fisher_information}
  &\Ieps = \int_{\T^n \times \R^n} \frac{ | \nabla_v \feps + (1 + \epsilon \steps /2) (v - \ueps) \feps |^2 }{\feps} \dv \dx .  
\end{align}
 Thanks to the uniform boundedness of the strength \eqref{e:seuniform}, we have 
  \begin{equation*}
   \int_{\T^n \times \R^n} \steps |v - \ueps|^2 \feps \dv \dx \lesssim \rme(\feps | \ueps),
\end{equation*}
and $\rme(\feps | \ueps) \leq \cE_\e \leq 2 \cH_\e + C \leq C$, as we concluded earlier so this term is bounded.

Turning to the macroscopic relative entropy,
\begin{equation*}
\begin{split}
  \label{eqn:Geps}
  \ddt \Geps 
  	&= \int_{\T^n} [ \nabla \u: \Reps - \rhoeps (\ueps - \u) \cdot \nabla \u \cdot (\ueps - \u) ] \dx
    + \frac{1}{\epsilon} \int_{\T^n} \rhoeps (\ueps_{\delta} - \ueps) \cdot \u \dx \\
    & \quad + \|\ueps\|_{L^2(\kappaeps)}^2 - (\ueps, \uavgeps)_{\kappaeps} + A , 
    \end{split}
\end{equation*}
where 
\[
A = (\ueps - \u, \uavgeps - \ueps)_{\kappa_{\rhoeps}} + \int_{\T^n} \rhoeps(\u - \ueps) \cdot \st (\uavg - \u) \dx 
\]
is the same alignment term as in the monokinetic case, which we will estimate using the original inheritance bounds \eqref{e:reg2'}.
We have
\begin{equation*}\label{ }
A \lesssim \int_{\T^n} \rhoeps |\ueps - \u|^2 \dx + W_1^2(\rhoeps, \rho) + \int_0^t \left[  \int_{\T^n} \rhoeps |\ueps - \u|^2 \dx + W_1^2(\rhoeps, \rho) \right] \dtau.
\end{equation*}
Let us note that 
\[
\int_{\T^n} \rhoeps |\ueps - \u|^2 \dx + W_1^2(\rhoeps, \rho) \leq \cH(\feps | \u).
\]
Hence,
\begin{equation*}\label{ }
A \lesssim  \cH(\feps | \u) + \int_0^t \cH(\feps | \u) \dtau.
\end{equation*}

We observe that when adding up the equations for $\Geps$ and $\Heps$, the energy term $\|\ueps\|_{L^2(\kappaeps)}^2 - (\ueps, \uavgeps)_{\kappaeps}$ cancels. Next, thanks to the assumed smoothness of $\u$, 
\begin{align*}
  \Big| \int_{\T^n} \rhoeps (\ueps - \u) \cdot \nabla \u \cdot (\ueps - \u) \dx \Big| 
    \lesssim \int_{\T^n} \rhoeps |\ueps - \u|^2 \dx \leq \cH(\feps | \u) .
\end{align*}
The local alignment term is the same as in the monokinetic case and it is handled by \cite[(9.14)]{S-EA}: 
\begin{align*}
  \frac{1}{\epsilon} \int_{\T^n} \rhoeps (\ueps_{\delta} - \ueps) \cdot \bu \dx \lesssim \frac{\delta}{\epsilon} .
\end{align*}

  The Reynolds stress term can be written
  \begin{align*}
    \Reps &= \int_{\R^n} [2 \nabla_v \sqrt{\feps} + (1 + \epsilon \steps /2) (v - \ueps) \sqrt{\feps}]
      \otimes [(v - \ueps) \sqrt{\feps}] \dv \\
      & \quad - \epsilon \steps/2 \int_{\R^n} (v - \ueps) \otimes (v - \ueps) \dv  .
  \end{align*}
 Once again due to the uniform boundedness of the strength \eqref{e:seuniform}, we have 
  \begin{equation*}\label{ }
   \Reps \lesssim \sqrt{\rme(\feps | \ueps) \Ieps} + \epsilon \rme(\feps | \ueps),
\end{equation*}
and as before, $\rme(\feps | \ueps)$ is bounded so
\[
\Reps \lesssim \sqrt{ \Ieps} + \epsilon \leq  \frac{1}{2\epsilon} \Ieps + 2 \epsilon.
\]
In total, we get 
  \begin{equation*}
  \begin{split}
    \frac{d}{dt} \Geps 
    	&\leq c_1 \cH(\feps | \mu) + c_2 \int_0^t \cH(\feps | \u) \dtau+  \frac{1}{2\epsilon} \Ieps  + c_3(\epsilon + \frac{\delta}{\epsilon}) \\
    & \quad + \|\ueps\|_{L^2(\kappaeps)} - (\ueps, \uavgeps)_{\kappaeps}.
    \end{split}
  \end{equation*}
  Combining the estimates on $\Heps$ and $\Geps$  we arrive at 
  \begin{align*}
    \frac{d}{dt} \cH(\feps | \mu) \lesssim \cH(\feps | \mu) +\int_0^t \cH(\feps | \mu) \dtau+ \epsilon + \frac{\delta}{\epsilon}.
  \end{align*}
The Gr{\"o}nwall Lemma concludes the proof. 
\end{proof}


\section{Relaxation to Global Maxwellian in 1D}
\label{sec:relaxation}

Adding noise with the strength coefficient to the velocity equation leads to the Fokker-Planck-Alignment system given by
\begin{align}
  \label{SM_kinetic_relaxation}
  \begin{cases}
      \partial_t f + v \cdot \nabla_x f + \nabla_v \cdot (\st (v - \uavg) f) = \sigma \st \D_v f \\
      \partial_t \st + \nabla_x \cdot (\st \uavg) = 0.
  \end{cases}
\end{align}
We study the case where the velocity averaging is given by the Favre averaging \eqref{Favre_averaging} (i.e. the $\wt$-model). 
Recall that the weight $\wt$ is defined by $\st = ( \rho \ast \phi ) \wt $ and the system becomes 
\begin{align}
  \label{e:FPA}
  \begin{cases}
      \partial_t f + v \cdot \nabla_x f = ( \rho \ast \phi ) \wt \n_v \cdot(  \sigma  \nabla_v f - v f) + \wt ((\u\rho) \ast \phi) \cdot \nabla_v f \\
      \partial_t \wt + \u_F \cdot \nabla_x \wt  = 0.
  \end{cases}
\end{align}

Let us discuss the well-posedness of \eqref{e:FPA}.  The general well-posedness theory for kinetic alignment equations based on a predefined strength $\st_\rho$ has been developed in \cite{S-EA}.  
The new system, under the uniform regularity assumptions \eqref{reg1_kernel} and \eqref{reg2_kernel}, which for the Favre filtration (with $\Phi_{\rho}(x,y) = \phi(x-y) / \rho\ast {\phi}(x)$) means all-to-all interactions \eqref{e:nondeg},
falls directly into the same framework. In fact, in this case every flock is automatically ``thick,'' meaning that 
\[
\inf_{\O} \rho \ast \phi = c_0 >0.
\]
We also have the a priori bounded energy 
\[
\dot{\cE} \leq c_1 \cE + c_2, \hspace{5mm} \cE(t) = \frac{1}{2} \int_{\Omega \times \R^n} |v|^2 f_t(x,v) \dx \dv
\]
on any finite time interval. 
Indeed, multiplying \eqref{e:FPA} by $\frac{1}{2}|v|^2$ and integrating by parts, we have: 
\begin{equation*}
\begin{split}
  &\ddt \frac{1}{2} \int_{\Omega \times \R^n} |v|^2 f \dx \dv \\
    &= \int_{\Omega \times \R^n} (\rho \ast \phi) \wt  v \cdot (vf - \sigma \nabla_v f) \dv \dx - \int_{\Omega \times \R^n} \wt ((\u \rho) \ast \phi) \cdot v f \dv \dx  \\
    &= \int_{\Omega \times \R^n} (\rho \ast \phi) \wt  \cdot |v|^2 f \dv \dx + n \sigma \int_{\Omega \times \R^n} (\rho \ast \phi) \wt  f \dv \dx \\
    & \quad - \int_{\Omega \times \R^n} \wt ((\u \rho) \ast \phi) \cdot v f \dv \dx  \\
    &= \int_{\Omega} (\rho \ast \phi) \wt  |\bu|^2 \rho \dx + n \sigma \int_{\Omega} (\rho \ast \phi) \wt  \rho \dx - \int_{\Omega} \wt ((\u \rho) \ast \phi) \cdot (\bu \rho) \dx  .
    \end{split}
\end{equation*}
Since $\wt$ remains bounded, the first term bounds the energy and the second term is bounded. 
For the last term, we estimate, 
\begin{align*}
  \int_{\Omega} \wt ((\u \rho) \ast \phi) \cdot (\bu \rho)  \dx \leq \wt_+ \phi_+ \|\bu\|_{L^1(\rho)}^2 \lesssim \wt_+ \phi_+ \|\bu\|_{L^2(\rho)}^2 = \wt_+ \phi_+ \cE,
\end{align*}
which yields the desired energy inequality. 
Moreover, since $\|\bu\|_{L^1(\rho)} \lesssim \|\bu\|_{L^2(\rho)} = \sqrt{\cE}$, we have the inheritance of regularity: 
\[
\| \p_x^k \u_F \|_\infty < C(k,J,T),
\]
for any $k\in \N$.  Consequently, according to (a trivial adaptation of) \cite[Theorem 7.1]{S-EA} for any data in the classical weighted Sobolev spaces $f_0\in H^{m}_q(\domain)$ and $\wt_0 \in C^m$, $\inf_{\O} \wt_0 >0$, $q>m$, where
\begin{equation*}\label{e:Sobdef}
H^{m}_q(\domain) =  \left\{ f :   \sum_{ k + l \leq m} \sum_{|\bk| = k, |\bl| = l}  \int_\domain | \jap{v}^{q - k - l}  \p^{\bk}_{x} \p^{\bl}_v f |^2 \dv \dx <\infty \right\},
\end{equation*}
there exists a unique global in time classical solution in the same space.  Let us note that the transport of $\wt$ preserves the lower bound on $\wt$ uniformly in time, and by the automatic thickness condition we have the ellipticity coefficient $(\rho \ast \phi) \wt$ uniformly bounded away from zero as well.  

Let us now recall a set of conditions on a given global solution $f$ that are sufficient to imply the global relaxation of $f$ to the Maxwellian
\begin{align*}
  \mu_{\sigma, \bar{\u}} = \frac{1}{|\T^n| (2\pi \sigma)^{n/2}} e^{-\frac{|v-\bar{\u}|^2}{2\sigma}},
\end{align*}
where $\bar{\u} = \int_\O \bu \rho \dx$. We note that this total momentum $\bar{\u}$ is not  preserved in time generally, but rather satisfies the equation
\begin{equation}\label{eqn:avg_velocity_evolution}
\partial_t \bar{\u} = \int_\O (\u_F - \u) \dk_\rho, \qquad \dk_{\rho} = \st \rho \dx.
\end{equation}
We can't determine whether $\bar{\u}$ eventually settles to a fixed vector, as it does for all of the classical alignment models including those that do not preserve the momentum. 

According to \cite[Proposition 8.1]{S-EA} a given solution $f$ converges to the Maxwellian in the sense defined in \eqref{e:relax} provided there exists a set of fixed constants $c_0,... >0$ such that 
\begin{enumerate} [(i)]
  \item  $c_0 \leq \st \leq c_1$ and $\|\nabla \st\|_{\infty} \leq c_2$,
  \item 
  \begin{align}
    \label{e:spgap}
    \sup  \Big\{ (\u, \uavg)_{\kappa_{\rho}} : \u \in L^2(\kappa_{\rho}), \bar{\u} = 0, \|\u\|_{L^2(\kappa_{\rho})} = 1 \Big\} \leq 1 - \epsilon_0,
  \end{align}
  \item $\|\st [\cdot]_{\rho}\|_{L^2(\rho) \to L^2(\rho)} + \|\nabla_x (\st [\cdot]_{\rho})\|_{L^2(\rho) \to L^2(\rho)} \leq c_3$.  
\end{enumerate}

We will be able to show (i)-(iii) in one dimensional case. 
Let us discuss these conditions starting from (i).  The key condition here is $\| \partial_x \wt \|_\infty \leq c_2$.  We can establish such uniform control in 1D only. Indeed, since
\[
\p_t \partial_x \wt + \partial_x (u_F \partial_x \wt) = 0,
\]
we can see that $\partial_x \wt$ satisfies the same continuity equation as $\rho \ast \phi$ (this only holds in 1D). Thus, 
\[
\p_t \frac{\partial_x \wt}{\rho \ast \phi} + u_F \p_x \frac{\partial_x \wt}{\rho \ast \phi} = 0,
\]
Since initially $\frac{|\partial_x \wt|}{\rho \ast \phi} \leq C$ for some large $C>0$ due to the all-to-all interaction assumption, this bound will persist in time.  
Hence, $\| \partial_x \wt \|_\infty \leq C \| \rho \ast \phi\|_\infty \leq c_2$.

Condition (iii) follows from (i). Indeed, since $\wt, \partial_x \wt$ remain uniformly bounded, it reduces to 
\[
\int_\O | (u\rho) \ast \phi |^2  \rho \dx + \int_\O | (u\rho) \ast {\partial_x \phi} |^2  \rho \dx \leq c_3 \|u\|^2_{L^2(\rho)}.
\]
This follows from the \HI.

Finally, let us address (ii). We reduce the computation of the spectral gap to the low-energy method for the classical Cucker-Smale model.  
To this end, we assume that $\phi = \psi \ast \psi$ where $\psi>0$ is another positive kernel. In other words, $\phi$ is Bochner-positive.    
In order to establish \eqref{e:spgap} it suffices to show that 
\begin{equation}
\label{ineq:suff_spgap}
(u, u)_{\k_\rho} - (u, \ave{u}_\rho)_{\k_\rho} \geq \epsilon_0 (u, u)_{\k_\rho}.
\end{equation}
Let us start  by symmetrizing and using cancellation  in the second obtained integral:
\begin{equation*}\label{}
\begin{split}
(u, u)_{\k_\rho} - (u, \ave{u}_\rho)_{\k_\rho} & = \int_{\O \times \O} u(x) \cdot (u(x) - u(y)) \rho(x) \rho(y) \phi(x-y) \wt(x) \dy \dx \\
&= \frac12 \int_{\O \times \O} |u(x) - u(y)|^2 \rho(x) \rho(y) \phi(x-y) \wt(x) \dy \dx \\
& \quad - \frac12 \int_{\O \times \O}  | u(y)|^2 \rho(x) \rho(y)(\wt(x) - \wt(y)) \phi(x-y) \dy \dx \\
&= I + II.
\end{split}
\end{equation*}
Notice that
\begin{equation*}
\begin{split}
	I &\geq \wt_-  \int_{\O \times \O} u(x) \cdot (u(x) - u(y)) \rho(x) \rho(y) \phi(x-y) \dy \dx \\
	&= \wt_- [ (u, u)_{(\rho \ast \phi) \rho} - (u, \ave{u}_\rho)_{(\rho \ast \phi) \rho} ].
\end{split}
\end{equation*}

The difference of the inner products inside the bracket represents exactly the spectral gap of the Cucker-Smale model computed in \cite[Proposition 4.16]{S-EA}. From that computation it follows that 
\[
(u, u)_{(\rho \ast \phi) \rho} - (u, \ave{u}_\rho)_{(\rho \ast \phi) \rho} \geq c  (u, u)_{(\rho \ast \phi) \rho}\geq c \wt_-  (u, u)_{\k_\rho},
\]
where $c$ depends only on the kernel $\psi$. 

Let us now estimate $II$:
\[
II \leq (\wt_+ - \wt_-)  (u, u)_{(\rho \ast \phi) \rho} \leq \frac{\wt_+ - \wt_-}{\wt_-} (u, u)_{\k_\rho}.
\]
We can see that provided
\[
\frac{\wt_+ - \wt_-}{\wt_-}  \leq \frac12 c \wt_-,
\]
we obtain \eqref{ineq:suff_spgap}, which is the desired result. 

Let us now collect all of the assumptions we have made and state the main result.

\begin{theorem}
  \label{thm:relaxation_wm}
  Suppose $n=1$ and the kernel is Bochner-positive, $\phi = \psi \ast \psi$, with $\inf \psi >0$.  Then any initial distribution $f_0\in H^{m}_q(\domain)$ and strength $\wt_0 \in C^m$ satisfying the following small variation assumption
  \begin{equation*}\label{ }
\frac{\sup (\wt_0) - \inf (\wt_0)}{\inf (\wt_0)^2}  \leq c,
\end{equation*}
for some absolute $c>0$, gives rise to a global classical solution $f,\wt$ which relaxes to the Maxwellian exponentially fast:
\begin{equation}\label{e:relax}
    \|f(t) - \mu_{\sigma, \bar{u}}\|_{L^1(\T^n \times \R^n)} \leq c_1 \sigma^{-1/2} e^{-c_2 \sigma t}.
\end{equation}
\end{theorem}

We note once again that the solution relaxes to a moving Maxwellian centered around a time-dependent momentum $\bar{u}$. 
There are two conceivable mechanisms of stabilizing $\bar{u}$. One is sufficiently fast alignment:
\begin{equation}\label{e:fastalign}
\int_0^\infty \sup_{x,y} |u(x,t)- u(y,t)| \dt <\infty,
\end{equation}
which our relaxation seems to be not strong enough to imply. 
Another is stabilization of the density to a uniform distribution $\rho \to \frac{1}{|\O|}$, which we do have from \eqref{e:relax}  and in the case when $\wt \equiv \const$ it does imply exponential stabilization of the momentum for Favre-based models, see \cite{S-EA}. 
If $\wt$ varies, even if  $\rho = \frac{1}{|\O|}$, from \eqref{eqn:avg_velocity_evolution} we have,
\[
 \partial_t \bar{u} = \int_\O (u_F - u) \dk_\rho = \frac{1}{|\O|} \int_\O ( u \ast\phi - u \|\phi\|_{L^1}) \wt \dx.
\]
We can see that unless \eqref{e:fastalign} holds, the persistent variations of $\wt$ may keep this term large leaving the possibility of a forever moving momentum $\bar{u}$.

\section*{Acknowledgments}
This work was  supported in part by NSF
grant  DMS-2405326 and the Simons Foundation.









\medskip
\medskip

\end{document}